\documentclass{article}
\usepackage{graphicx}
\usepackage{fullpage}
\usepackage{amsmath}
\usepackage{amssymb}
\usepackage{verbatim} %multiline comments

\usepackage{tikz}

\usepackage{ulem} %strikethrough \sout{}

\usepackage{bbm} %blackboard bold for numbers
\usepackage{amsthm} % theorems

\newtheorem{thm}{Theorem}[section]
\newtheorem{defn}[thm]{Definition}
\newtheorem{lem}[thm]{Lemma}
\newtheorem{prop}[thm]{Proposition}
\newtheorem{cor}[thm]{Corollary}

\title{Mis\`ere Hackenbush Flowers}

\author{Irene Y.~Lo}

\begin{document}

\maketitle
\normalem

\begin{abstract}
We show that any disjunctive sum of Hackenbush Flowers $G$ has as evil twin $G^* \in \{G, G+*\}$ such that the outcomes of $G$ under normal and mis\`ere play are the same as the outcomes of $G^*$ under mis\`ere and normal play respectively. We also show that, under mis\`ere play, any Green Hackenbush position that has a single edge incident with the ground is equivalent to a nim-heap.
\end{abstract}

\section{Introduction}
\textit{Hackenbush} is a combinatorial game played on a finite graph with colored edges. In Regular Hackenbush, the edges can be colored red, blue or green, in Red-Blue Hackenbush, the edges can be colored red or blue, and in Green Hackenbush, all the edges are green. A subset of the vertices is designated the ground and usually drawn on a long horizontal line. The players Left and Right take turns cutting edges, with the restriction that Left can only cut blue and green edges, and Right can only cut red and green edges. At any stage in the game, any component of the graph no longer connected to the ground is discarded. Under \textit{normal-play}, the last player to move wins, and under \textit{mis\`ere-play}, the last player to move loses. Figure \ref{hackflowers} shows some examples of games in Hackenbush.

		\begin{figure}[h!]
		\centering

\begin{tikzpicture}
\draw[line width=1mm] (-3,0)--(-1,0);

\draw[line width=1mm,draw=green] (-2,0)--(-2,1);
\draw[fill=white] (-2,0) ellipse (4pt and 4pt);
\draw[fill=white] (-2,1) ellipse (4pt and 4pt);
\draw (-2,-0.5)--(-2,-0.5)  node {(a) The game $*$};

\end{tikzpicture}
\begin{tikzpicture}
\draw[line width=1mm] (0,0)--(4,0);

\draw[line width=1mm,draw=green] (2,0)--(2,3)--(1,4);
\draw[line width=1mm,draw=green] (2,1)--(2,3)--(3,4);
\draw[line width=1mm,draw=green] (0,5)--(1,4);
\draw[line width=1mm,draw=green] (2,5)--(1,4);
\draw[line width=1mm,draw=green] (0,5)--(1,6);
\draw[line width=1mm,draw=green] (2,5)--(1,6);

\draw[fill=white] (0,5) ellipse (4pt and 4pt);
\draw[fill=white] (2,5) ellipse (4pt and 4pt);
\draw[fill=white] (1,6) ellipse (4pt and 4pt);

\draw[fill=white] (2,0) ellipse (4pt and 4pt);
\draw[fill=white] (2,1) ellipse (4pt and 4pt);
\draw[fill=white] (2,2) ellipse (4pt and 4pt);
\draw[fill=white] (2,3) ellipse (4pt and 4pt);
\draw[fill=white] (1,4) ellipse (4pt and 4pt);
\draw[fill=white] (3,4) ellipse (4pt and 4pt);

\draw (2,-0.5)--(2,-0.5)  node {(b) a Shrub};

\end{tikzpicture}
\begin{tikzpicture}
\draw[line width=1mm] (3,0)--(5,0);

\draw[line width=1mm,draw=green] (4,0)--(4,3);
	%L
\draw[draw = red, line width = 1mm] (4-0.64,1+3-0.19) arc (70:210:0.25cm);%
\draw[draw = red, line width = 1mm] (4-0.94,1+2.45) arc (210:235:0.6cm);%
\draw[draw = red, line width = 1mm] (4-0.78,1+2+0.26) arc (235:265:1.5cm);
\draw[draw = red, line width = 1mm] (4-0.42,1+2.5+0.17) arc (45:70:0.6cm);
\draw[draw = red, line width = 1mm] (4-0.03,1+2.0+0.00) arc (15:45:1.5cm);%
	%L bottom
\draw[draw = red, line width = 1mm] (3+0.06,1+2-0.49) arc (150:290:0.25cm); %
\draw[draw = red, line width = 1mm] (4-0.65,1+1.14) arc (290:315:0.6cm);
\draw[draw = red, line width = 1mm] (4-0.43,1+1+0.28) arc (315:345:1.5cm);
\draw[draw = red, line width = 1mm] (4-0.76,1+1.5+0.21) arc (125:150:0.6cm);%
\draw[draw = red, line width = 1mm] (4-0.03,1+2.0-0.02) arc (95:125:1.5cm);%

	%R
\draw[draw = red, line width = 1mm] (4+0.93,1+3-0.53) arc (-20:110:0.25cm);%
\draw[draw = red, line width = 1mm] (4+0.75,1+2.28) arc (-55:-30:0.6cm); %
\draw[draw = red, line width = 1mm] (4+0.76,1+2+0.28) arc (305:275:1.5cm);%
\draw[draw = red, line width = 1mm] (4+0.62,1+2.79) arc (110:135:0.6cm);%
\draw[draw = red, line width = 1mm] (4+0.41,1+2.0+0.67) arc (135:165:1.5cm);
	%R bottom
\draw[draw = red, line width = 1mm] (4+0.64,1+2-0.85) arc (-110:30:0.25cm); %
\draw[draw = red, line width = 1mm] (4+0.42,1+1.29) arc (225:250:0.6cm);%
\draw[draw = red, line width = 1mm] (4+0.04,1+2-0.05) arc (195:225:1.5cm);%
\draw[draw = red, line width = 1mm] (4+0.95,1+1.5+0.01) arc (30:55:0.6cm);%
\draw[draw = red, line width = 1mm] (4+0.78,1+2.0-0.30) arc (55:85:1.5cm);
%circles
\draw[fill=white] (4,0) ellipse (4pt and 4pt);
\draw[fill=white] (4,1) ellipse (4pt and 4pt);
\draw[fill=white] (4,2) ellipse (4pt and 4pt);
\draw[fill=white] (4,3) ellipse (4pt and 4pt);

\draw (4,-0.5)--(4,-0.5)  node {(c) a Flower};

\end{tikzpicture}
\begin{tikzpicture}
\draw[line width=1mm] (0,1)--(2,1);

\draw[line width=1mm,draw=green] (1,1)--(1,4);
\draw[line width=1mm,draw=red] (0,5)--(1,4);
\draw[line width=1mm,draw=blue] (2,5)--(1,4);
\draw[line width=1mm,draw=blue] (0,5)--(1,6);
\draw[line width=1mm,draw=red] (2,5)--(1,6);

\draw[fill=white] (0,5) ellipse (4pt and 4pt);
\draw[fill=white] (2,5) ellipse (4pt and 4pt);
\draw[fill=white] (1,6) ellipse (4pt and 4pt);

\draw[fill=white] (1,1) ellipse (4pt and 4pt);
\draw[fill=white] (1,2) ellipse (4pt and 4pt);
\draw[fill=white] (1,3) ellipse (4pt and 4pt);
\draw[fill=white] (1,4) ellipse (4pt and 4pt);

\draw (1,0.5)--(1,0.5)  node {(d) a Generalized Flower};

\end{tikzpicture}
\begin{tikzpicture}
\draw[line width=1mm] (6,0)--(8,0);

\draw[line width=1mm,draw=green] (7,0)--(7,1);
\draw[line width=1mm,draw=blue] (7,1)--(7,2);
\draw[line width=1mm,draw=red] (7,2)--(7,3);
\draw[line width=1mm,draw=blue] (7,3)--(7,6);
\draw[fill=white] (7,0) ellipse (4pt and 4pt);
\draw[fill=white] (7,1) ellipse (4pt and 4pt);
\draw[fill=white] (7,2) ellipse (4pt and 4pt);
\draw[fill=white] (7,3) ellipse (4pt and 4pt);
\draw[fill=white] (7,4) ellipse (4pt and 4pt);
\draw[fill=white] (7,5) ellipse (4pt and 4pt);
\draw[fill=white] (7,6) ellipse (4pt and 4pt);

\draw (7,-0.5)--(7,-0.5)  node {(e) a Sprig};

\end{tikzpicture}

		\caption{Some Hackenbush Game Positions}
		\label{hackflowers}
		\end{figure}
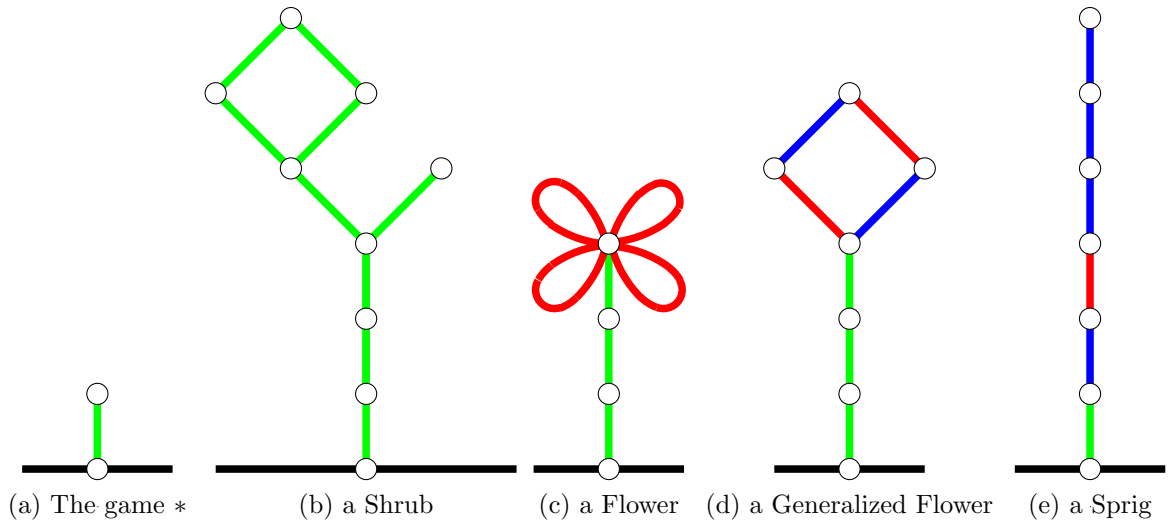

Regular Hackenbush, although simple to explain, is famously complex and theoretically deep. It is NP-hard under both play conventions \cite{np}, and many simple Hackenbush positions, such as Flowers, do not have known polynomial-time solutions under either play convention. Normal-play Hackenbush has been used by standard texts in combinatorial game theory to demonstrate many definitions and concepts in normal-play game theory (\cite{WW, ONAG}). However, many fundamental simplifications made in normal-play do not generalize to mis\`ere play, and much less is currently known about mis\`ere play Hackenbush and mis\`ere play combinatorial games in general (\cite{ONAG, canmis}).

In this paper, we will study Flowers and Green Hackenbush game positions with only a single edge incident with the group, some of which are shown in Figure \ref{hackflowers}. Some terms necessary to describe the particular games we are interested in are as follows. A \textit{String} is a path with one end on the ground. A \textit{Stalk} of \textit{height} $n$ is a green String of length $n$. The game $*$ is a Stalk of height $1$. A \textit{Shrub} is a rooted graph with green edges, such that the root vertex has degree one and is the only vertex on the ground. 

Let $g$ and $h$ be graphs with vertex subsets $gr(g)$ and $gr(h)$ designated to be the ground, and let them correspond to Hackenbush positions $G$ and $H$. We define the \textit{disjunctive sum} $G+H$ by the Hackenbush position corresponding to the disjoint union of $g \cup h$, with the set of ground vertices given by $gr(g) \cup gr(h)$. We say that a subgraph $g_1 \subset g$ of the underlying graph of a Hackenbush position \textit{supports} its complement $g_2 = g \backslash g_1 $ if removing any edge in $g_1$ disconnects $g_2$ from the ground, and we say that $G$ is the \textit{ordinal sum} of $G_1$ and $G_2$ \cite{WW}. 

A \textit{Flower} is a Stalk supporting a number of loops, all red or all blue, and a \textit{Generalized Flower} is a green String supporting a red-blue Hackenbush position. A \textit{Sprig} is green edge supporting a red-blue String, and a \textit{Generalized Sprig} is a green edge supporting a red-blue Hackenbush position. 

A \textit{star-based position} is the ordinal sum of $*$ and any position. The \textit{stem} of a star-based position is the longest induced green path from the grounded vertex of the underlying graph, and the \textit{height} of a star-based position is the length of the stem. Shrubs, Generalized Flowers and Generalized Sprigs are all examples of star-based positions in Hackenbush.

The outcome classes of sums of some of these game positions have been determined. Conway classified the outcome classes of Hackenbush Sprigs and Green Hackenbush under normal play in \cite{ONAG}. In particular, he showed that any Green Hackenbush games, such Hackenbush Shrubs, are equivalent to nim heaps. McKay, Milley and Nowakowski classified the outcome classes of Hackenbush Sprigs under mis\`ere play in \cite{sprigs}, and noted that they are the same as the normal play outcome classes after addition of a single green edge. These results suggest a close relationship between normal and mis\`ere play outcomes of larger classes of games in Hackenbush. 
%\textit{<--- should I define /where should I define impartial games and faithfully normal games?}

%%%%%%%%%%%%%%%%
%It is well known that the outcome under normal play of a game $G$ is equal to the outcome under mis\`ere play of the game $H$ obtained by attaching all the grounded vertices of $G$ to the top of a single green edge. 

%%%%%%%%%%%%%%%%

In this paper we focus on Shrubs, Generalized Sprigs and Generalized Flowers, and prove a generalization of McKay, Milley and Nowakowski's result on the relationship between the normal and mis\`ere play outcome classes. 

Our main theorem requires the following results and definitions (see also Section \ref{prelim}).

\begin{thm}\label{thm:shrubs}
Let $S$ be a Shrub with normal play nim-value $g^+(S)$. Then $S$ is equivalent in the mis\`ere universe to a Stalk of length $g^+(S)$.
\end{thm}

\begin{defn}
Let $G = \sum_{i=1}^n G_i $ be the disjunctive sum of $n$ Shrubs and Generalized Flowers. Let $G'_i$ be the Stalk equivalent to $G_i$ if $G_i$ is a Shrub, and just $G_i$ otherwise. Then the \textit{evil twin} $G^*$ of $G$ is given by
\begin{eqnarray*}
G^* = \left\{ 
\begin{array}{cc}
G & \text{ if $G'_i$ has height at least $2$ for some $i$} \\
G+* & \text{ if $G'_i$ has height $1$ for all $i$} 
\end{array}
\right.
\end{eqnarray*}
\end{defn} 

\noindent Our main result is the following.

\begin{thm}\label{thm:intro}
If $G$ is the disjunctive sum of Shrubs and Generalized Flowers, then the outcome of $G$ under normal play is the same as the outcome of $G^*$ under mis\`ere play, and vice versa.
\end{thm}

As a corollary, we can specify a winning strategy under mis\`ere play for disjunctive sum of Shrubs and Generalized Flowers in terms of the winning strategies under normal play:

\begin{cor}
Let $G$ be a game comprised of Green Shrubs and Generalized Flowers, and suppose that $G = G^*$. Then under mis\`ere play, the winning player's strategy is to play as in normal play, until their only winning move under mis\`ere play is to some game $G' = \sum_{i=1}^m G'_i$, where each $G'_i$ is a star-based position of height $1$. 
\end{cor}

Let $\mathcal{G}^*$ be the set of positions $G$ with evil twin $G^*=G+*$. Then the corollary follows from the fact that, by following the winning strategy outlined in the proof of Theorem \ref{thm:shrubs} the winning player can win by making each move to a game not in $\mathcal{G}^*$ until their only winning move is to a game in $\mathcal{G}^*$, and then making each move to a game in $\mathcal{G}^*$ for the rest of the game. We also provide specific winning strategies under mis\`ere play for such games $G' = \sum_{i=1}^m G'_i$, where each $G'_i$ is a star-based position of height $1$. This reduces mis\`ere play Hackenbush Shrubs and Flowers to understanding the normal play version, and generalizes similar results for the relationship between normal and mis\`ere play outcomes and strategies in nim (\cite{nim}) and Hackenbush Sprigs (\cite{sprigs}).

In Section \ref{Shrubs}, we prove Theorem \ref{thm:shrubs}, which says that any Shrub is equivalent to the same nim heap in both play conventions. In Section \ref{npGF}, we summarize some known results for normal play Hackenbush Flowers. In Section \ref{GenFlowers}, we specify the outcome classes of sums of Generalized Sprigs and Stalks (Theorem \ref{thm:Generalized Sprigs}) and show that Theorem \ref{thm:intro} applies for these games. We then show that Theorem \ref{thm:intro} holds for all sums of Generalized Flowers and Stalks that are not sums of Generalized Sprigs (Theorem \ref{thm:main}).
%Have to show that for just sprigs, G^* \neq G

\subsection{Preliminaries}\label{prelim}

We introduce here the concepts pertinent to the ensuing discussion. Relevant definitions can also be found in \cite{WW} and \cite{ONAG}.

A game $G$ is a position, defined recursively by $G = \{\mathcal{G}^L| \mathcal{G}^R\}$, where the set of left options $\mathcal{G}^L$ is the set of positions that Left can move to, and the set of right options $\mathcal{G}^R$ is the set of positions that Right can move to. Left and right options are also referred to in the literature as left and right followers respectively. A combinatorial game is \textit{impartial} if, at every stage, the set of left options is equal to the set of right options, and \textit{partizan} if they may differ.

In impartial combinatorial games, there are two outcome classes: \textit{Next-player win ($\mathcal{N}$)}; and \textit{Previous-player win ($\mathcal{P}$)}.  In partizan combinatorial games, there are two additional outcome classes: \textit{Left-player win ($\mathcal{L}$)}; and \textit{Right-player win ($\mathcal{R}$)}. Under both normal and mis\`ere play, the outcome classes are partially ordered as in Figure \ref{partialorder}. Better , or larger, outcomes under the partial ordering are more advantageous for the player Left.

		\begin{figure}[h!]
		\centering

\begin{tikzpicture}[scale=1.5]
\draw (1+0.3,0+0.3)--(2-0.3,1-0.3);
\draw (1-0.3,0+0.3)--(0+0.3,1-0.3);
\draw (1-0.3,2-0.3)--(0+0.3,1+0.3);
\draw (1+0.3,2-0.3)--(2-0.3,1+0.3);

\draw [draw=white] (1,2)--(1,2) node [midway] {$\mathcal{L}$};
\draw [draw=white] (1,0)--(1,0) node [midway] {$\mathcal{R}$};
\draw [draw=white] (0,1)--(0,1) node [midway] {$\mathcal{P}$};
\draw [draw=white] (2,1)--(2,1) node [midway] {$\mathcal{N}$};
\end{tikzpicture}

		\caption{The partial order of outcome classes}
		\label{partialorder}
		\end{figure}
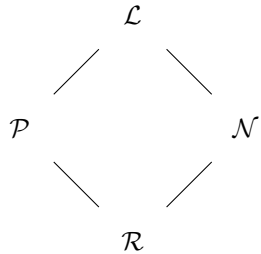

The outcome functions $o^+(G)$ and $o^-(G)$ map a game $G$ to its outcome class under normal and mis\`ere play respectively. If $o^+(G) = \mathcal{O}$ for $\mathcal{O}$ an outcome class, we write $G \in \mathcal{O}^+$, and similarly under mis\`ere play. For example, $o^+(0)=\mathcal{P}$, $o^-(0)=\mathcal{N}$ and $0 \in \mathcal{P}^+ \cap \mathcal{N}^-$.

The \textit{disjunctive sum} of two games $G$ and $H$ is defined recursively by $G+H = \{\mathcal{G}+\mathcal{H}^L, \mathcal{G}^L+\mathcal{H}  ~|~ \mathcal{G}+\mathcal{H}^R, \mathcal{G}^R+\mathcal{H}\}$. The \textit{ordinal sum} is defined recursively by $G:H = \{\mathcal{G}:\mathcal{H}^L, \mathcal{G}^L ~|~ \mathcal{G}:\mathcal{H}^R, \mathcal{G}^R\}$. Note that these definitions coincide with the earlier definitions of disjunctive and ordinal sums of Hackenbush games. 

Given an impartial combinatorial game $G$, its \textit{nim-value}, $g(G)$, is defined recursively as the smallest non-negative integer that is not the nim-value of an option of $G$, where $g(0)=0$ under normal play and $g(0)=1$ under mis\`ere play. To avoid confusion, we will let $g^+(G)$ denote the nim-value of a game under normal play, and let $g^-(G)$ denote the nim-value under mis\`ere play. It is well known that $G \in \mathcal{P}^+$ if and only if $g^+(G)=0$, and the normal-play canonical form of any impartial game $G$ is given by the nim heap $*_{g^+(G)}$ \cite{ONAG}. By contrast, little is known about the relationship between mis\`ere outcome classes and nim-values.

Two games $G$ and $H$ are said to be equal under normal play if $o^+(G+X) = o^+(H+X)$ for all games $X$, and equal under mis\`ere play if $o^-(G+X) = o^-(H+X)$ for all games $X$. We write $G =^+ H$ and $G=^-H$ for equality under normal and mis\`ere play respectively. Notions of inequality are similarly defined. We say $G \geq^+ H$ if $o^+(G+X) \geq o^+(H+X)$ for all games $X$, and $G \geq^- H$ if $o^-(G+X) \geq o^-(H+X)$ for all games $X$, where comparisons of outcome classes are with respect to the partial ordering in Figure \ref{partialorder}.

Under normal play, any left option $A$ of a game $G$ is said to be \textit{dominated} by another left option $B$ of $G$ if $A \leq^+ B$, and similarly any right option $C$ of a game $G$ is dominated by another right option $D$ of $G$ if $C \geq^+ D$. Any left option $A$ of a game $G$ is said to be \textit{reversible} if it has a right option $A^R$ such that $A^R \leq G$, and similarly any right option $C$ of $G$ is said to be reversible if it has a left option $C^L$ such that $C^L \geq G$. Similar definitions hold under mis\`ere play.
 
The canonical form of a game is obtained by removing all dominated and reversible options, and each game has a unique canonical form under each play convention (\cite{WW, ONAG, canmis}). Examples of canonical forms of games include the \textit{nim-heaps} $0=\{\cdot | \cdot\}$, $* =*_1= \{0|0\}$ and $*_n=\{0,*_1, \ldots, *_{n-1}|0,*_1, \ldots, *_{n-1}\}$, which correspond to nim piles of size $n$.

The \textit{negative} of a game $G$ is defined recursively by $-G = \{-\mathcal{G}^R|-\mathcal{G}^L\}$. Note that the negative of a game is its additive inverse in normal play, that is, $G+(-G)=0$, but this is not necessarily true in mis\`ere play. Hence, to avoid incorrect canceling, we will often write $\bar{G}$ to represent $-G$. The game $G-H$ is defined as the disjunctive sum of $G$ and $-H$.

Under normal play, there exist easy tests for equality and inequality: $G = 0 \text{ if and only if } G \in \mathcal{P}$, $G = H \text{ if and only if } G -H=0$, and $G > H \text{ if and only if } G -H>0$. However, under mis\`ere play, almost all known tests for equality and inequality require checking all possible games $X$ \cite{canmis}. To counter this, Plambeck \cite{misquo} introduced the notion of restricting the universe $\mathcal{X}$ to which $X$ may belong:
\vspace{-5pt}
\begin{eqnarray*}
G \equiv^- H \text{ (mod $\mathcal{X}$) if } o^-(G+X) = o^-(H+X) \text{ for all games } X \in \mathcal{X}\\
G \geq^- H \text{ (mod $\mathcal{X}$) if } o^-(G+X) \geq o^-(H+X) \text{ for all games } X \in \mathcal{X}.
\end{eqnarray*}

Some of the universes of interest to us are as follows. The universe $\mathcal{H}_*$ is the universe of all disjunctive sums of star-based Hackenbush positions. The universe $\mathcal{D}$ is the universe of all games such that, at any stage, either both players have a move or neither player has a move. These games are called \textit{dicot} games. We remark that $\mathcal{D}$ is closed under disjunctive sums, and $\mathcal{H}_* \subset \mathcal{D}$.
 
By restricting $\mathcal{X}$ to the universe of positions of a given combinatorial game, Plambeck and Siegel were able to relate the normal and mis\`ere play outcome classes of a large class of impartial mis\`ere games, and classify the outcome classes of many specific games. Using the same idea, Allen \cite{Allen} and McKey et al. \cite{sprigs} classified certain classes of partizan mis\`ere games. In particular, in \cite{Allen}, Allen showed that $*+*$ is equivalent to $0$ in the dicot universe. This is of interest, since $*+*$ is not equivalent to $0$ in any universe containing a game equivalent to a Hackenbush position with one blue edge.

For any $m, n \in \mathbb{N}$, their \textit{nim-sum} $m \oplus n$ is given by taking the bitwise xor. The \textit{upper nim-sum} is defined by $m ~\text{\sout{$\uparrow$}}~ n := 
		\mathop{\mbox{max}}_{0 \leq n' \leq n} m \oplus n'$ and the \textit{lower nim-sum} by 
$m ~\text{\sout{$\downarrow$}}~ n := 
		\mathop{\mbox{min}}_{0 \leq n' \leq n} m \oplus n'$
When taking the composition of nim-sum operations, we will always let the order of operations be from left to right.

It is simple to verify that the upper and lower nim-sums $m ~\text{\sout{$\uparrow$}}~ n$ and $m ~\text{\sout{$\downarrow$}}~ n$ can also be computed using the binary expansions of $m$ and $n$, as in the following lemma.

%lem: up and down
\begin{lem}[{Berlekamp, \cite{notebook}}]\label{lem:updown} Let $m,n \in \mathbb{N}$ have binary expansions $m = \sum_{i=0}^l a_i 2^i$ and $n = \sum_{i=0}^l b_i2^i$, $a_i, b_i \in \{0,1\}$, where $l$ is the largest integer such that either $a_l=1$ or $b_l=1$. Let $\alpha$ be the largest integer $j$ such that $a_j = b_j = 1$, and if no such integer exists let $\alpha=-1$. Let $\beta$ be the largest integer $j$ such that $a_j=0, b_j=1$, and if no such integer exists let $\beta = -1$. Then
\begin{eqnarray*}
m ~\text{\sout{$\uparrow$}}~ n = \sum_{i=0}^{\alpha} 2^i + \sum_{i=\alpha+1}^l a_i \oplus b_i \cdot 2^i, 
~~~~~ m ~\text{\sout{$\downarrow$}}~ n = \sum_{i=\beta+1}^l a_i \oplus b_i \cdot 2^i. 
\end{eqnarray*}
\end{lem}

It follows from this formulation of the upper and lower nim-sum that $m ~\text{\sout{$\uparrow$}}~ n$ is commutative and monotonically increasing in both variables and $m ~\text{\sout{$\downarrow$}}~ n$ is monotonically increasing in $m$ and monotonically decreasing in $n$.

Nim sums are essential to the analysis of combinatorial games. In particular, Bouton showed that a nim game is a previous-play win position if and only if the nim-sum of the pile sizes is $0$ \cite{nim}; Sprague and Grundy independently showed that if $G = \sum_{i=1}^n G_i$ is a disjunctive sum of games, $g(G)$ is given by the nim-sum of the $g(G_i)$ \cite{grundy, sprague}; and Berlekamp showed that normal play Hackenbush flowers can be reduced to understanding the behavior of upper and lower nim-sums \cite{notebook}.

%NEED TO REDEFINE MISERE PLAY CANONICAL FORM
%IN PARTICULAR, REVERSIBLE MOVES
% G --> H if for all options G' of G not options of H, H is option of G'
	% suppose G --> *_m (wrong use of Conway's lemma)
	% then for all options G' of G not *_k, k<m, *_m is option of G'
	% hence g(G') \neq m
	% hence g(G) \leq m
	% but *_m has less options than G, so g(G) \geq *_m
	% so g(G)=m

%Fixing conway's lemma?
	% if G reducible, then reducible to option of option, which has less edges, so has canonical form *_m
	% otherwise not reducible, all options reducible to nim heap, g values 0, 1, 2, .., m_1, m_1, m_2 ...  m_i>m
	% consider *_m. all options of *_m are options (or reducible from options) of G, all options of G not reducible to	 options of *_m have *_m as option
	% if H reducible from option G' of G, G' has same options as H except some additional reversible ones

%Shrubs
\section{Hackenbush Shrubs}\label{Shrubs}
Let $\mathcal{S}$ be the class of all disjunctive sums of Shrubs. In this section, we will show that under mis\`ere play, any Shrub $S$ is equivalent to the nim heap $*_{g^+(S)}$ (Theorem \ref{thm:shrubs}). It follows that the height of any shrub $S$ is given by its normal-play nim-value, $g^+(S)$. Together with known results for nim, this allows us to determine the relationship between the outcomes of $S$ and $S^*$ for any $S \in \mathcal{S}$.

It can easily be shown that the normal-play nim-value of a disjunctive sum of nim heaps $\sum_{i=1}^n *_{m_i}$ is given by the nim-sum of their nim-values, $g^+(\sum_{i=1}^n *_{m_i}) = m_1 \oplus \cdots \oplus m_n$. It is also well known, as shown by Bouton in \cite{nim}, that normal play nim and mis\`ere play nim have the same outcome classes and winning strategies, except when the only winning move is to a position where all the piles are of size $1$. In particular, if we represent a nim pile of size $n$ with a Hackenbush Stalk of height $n$ and define evil twins of games in the same manner, then $o^+(G)=o^-(G^*)$ and $o^+(G^*)=o^-(G)$ for all games of nim.

\begin{lem}[Bouton \cite{nim}]\label{lem:nim}
Let $G = \sum_{i=1}^n *_{k_i}$ be a game of nim, where $n \geq 0$ and $k_i$ are positive integers. Let $G^* = G+*$ if all $k_i$ are equal to $1$ and $G^*=G$ otherwise. Then
\begin{eqnarray*}
o^+(G) = o^-(G^*) = 
\left\{
\begin{array}{cc}
\mathcal{N} & k_1 \oplus \cdots \oplus k_n \neq 0 \\ 
\mathcal{P} & k_1 \oplus \cdots \oplus k_n = 0.
\end{array}
\right.
\end{eqnarray*}
\end{lem}

Note in particular that from this, the equality $o^+(G^*)=o^-(G)$ also easily holds, as $(G^*)^*$ is either just $G$ itself, or it equals $G+*+*$, in which case $*+*$ can be canceled to give the required result. We remark also that if $G^* \neq G$ then $o^{\pm}(G^*) \neq o^{\pm}(G)$. 

In order to prove the equivalence of Shrubs and nim piles, we first state an essential result in evaluating normal play nim-values that will be useful in our proof of the main theorem. It is a special case of the `Colon Principle' outlined by Berlekamp, Conway and Guy.

\begin{lem}[Berlekamp, Conway and Guy, \cite{WW}]\label{lem:colon}
	Let $G$ and $H$ be impartial Hackenbush games. Then $g^+(G:H) = g^+(G)+g^+(H)$. 
\end{lem}

Using Lemma \ref{lem:colon}, we may prove the following restatement of Theorem \ref{thm:shrubs}.

\begin{thm}\label{shrubs2}
Let $S$ be a Shrub with normal play nim-value $n=g^+(G)$. Then for any game $T$ in the entire mis\`ere universe,
\begin{eqnarray*}
o^-(S+T) = o^-(*_n+T).
\end{eqnarray*}
\end{thm}
\begin{proof}
We prove this by induction on $n$.

\begin{itemize}
\item \textit{Base case: $n=0$}\\
We show that $S$ must be empty. Otherwise, $S=*:\hat S$ for some Green Hackenbush game $\hat S$, so by Lemma \ref{lem:colon}, the normal play nim-value of $S$ is $0=n=1+g^+(\hat S) \geq 1$, which is a contradiction. Hence $S = 0$ identically.
\item \textit{Inductive step}\\
Suppose that Left wins $*_n + T$. Then she has a winning strategy in $S+T$ as follows. If Right has just played in $S$ to move to some option $S^R$ with a larger nim-value $g^+(S^R)>n$, then Left plays in $S^R$ to move it to some option $S^{RL}$ with nim-value $g^+(S^{RL}) = n$. This is possible as, by definition of the nim-value, for any $m < g^+(S^R)$, $S^R$ has a follower with nim-value $m$. Otherwise, Left mimics her winning strategy in $*_n + T$, treating any game $S'$ in the game tree of $S$ with normal play nim-value $n' \leq n$ just as the nim heap $*_{n'}$. Eventually one of the players moves in the $S$ portion of the game to bring it to some game $S''$ with normal play nim-value $n'' < n$, and the resulting overall position is of the form $S''+T''$. If this player is Left, then Left wins $*_{n''}+T''$ playing second, so by induction Left wins $S''+T''$ playing second. If this player is Right, then the previous position was of the form $S' + T''$ where $g^+(S')=n$ and Left wins $*_n + T''$ playing second. Hence Left wins $*_{n''}+T''$ playing first, and by induction Left wins $S''+T''$ playing first. Thus Left wins.
\end{itemize}
\end{proof}

We note that Theorem \ref{shrubs2} means that the height of a Shrub $S$ is given by its normal play nim-value, as the Stalk of height $g^+(S)$ is equivalent to $S$, and if any Shrub $S'$ equivalent to $S$ has height greater than $g^+(S)$ then it also has normal play nim-value greater than $g^+(S)$, which is not possible as equivalence under mis\`ere play implies equivalence under normal play. We remark also that if Left has just moved to a game $G$ such that $G^*=G+*$ and Left can win, then by following the strategy specified above, she can win by only moving to games $G'$ such that $G'^* = G'+*$ throughout the rest of the game.  

With this in mind, the relationships between the outcomes of a disjunctive sum of Shrubs and those of its twin follow as a simple corollary of Theorem \ref{thm:shrubs}.

\begin{thm}\label{prop:shrubs}
Let $S \in \mathcal{S}$ be a disjunctive sum of Shrubs. Then
\begin{eqnarray*}
o^+(S) = o^-(S^*), \text{ and } o^+(S^*)=o^-(S).
\end{eqnarray*}
\end{thm}
\begin{proof}
Taking the equivalent nim heap for each Shrub, as in Theorem \ref{thm:shrubs}, and comparing with Lemma \ref{lem:nim} yields the result. 
\end{proof}

\section{Normal Play Generalized Flowers}\label{npGF}

Let $\mathcal{F}$ be the class of all disjunctive sums of Generalized Flowers and Stalks. The \textit{Blossom} of a Generalized Flower is the Red-Blue Hackenbush game supported by the stem. 

In normal play, a finite Red-Blue Hackenbush game $G$ has a corresponding \textit{value} $x = \frac{a}{2^b}$ \cite{ONAG}. This is the dyadic number obtained recursively as the unique number such that $b$ is minimal, $a$ is odd and $x_L < x < x_R$, where $x_L$ is the largest value over all left options of $G$, and $x_R$ is the smallest value over all right options of $G$, given that the empty game has value $0$. For example, the Red-Blue Hackenbush game consisting of a single blue edge, $\{0| \cdot\}$, has value $1$, a single red edge $\{\cdot, 0\}$ has value $-1$, and a game consisting of both a blue and red edge, $\{-1|1\}$, has value $0$.

Such a value is well-defined for all Red-Blue Hackenbush games, and is sufficient to classify the outcome class of the game: a game $G$ with value $x$ is a Left player win position for $x>0$, a Right player win position for $x<0$ and a Previous player win position for $x=0$ (\cite{WW, ONAG}). If $x>0$ the largest left option has non-negative value and all right options also have positive value; if $x<0$ the smallest right option has non-positive value and all left options have negative value; and if $x=0$ all left options have negative value and all right options have positive value.

Under both play conventions, a Flower has canonical form $*_h:x$, where $h$ is a positive integer and $x$ is an integer. If $x>0$ the Flower is \textit{blue}, if $x<0$ the Flower is \textit{red}, and if $x=0$ the Flower is a Stalk. We note that in normal play, if $x+y=0$, then $*_h:x=\overline{*_h:y}$ so that $(*_h:x)+(*_h:y)=0$. Hence in normal play we may assume that no red Flower and blue Flower have the same height and blossom value.

The outcome under normal play of a game of Flowers is a long-standing question by Berlekamp, which can be phrased as follows (\cite{WW}).

\textit{Who wins $\sum_{i=1}^n *_{h_i}:x_i$ where $h_i$ are positive integers and $x_i$ integers?}

The following results for normal play Flowers are outlined in \cite{WW} as well as \cite{notebook} and will be useful in our discussion of mis\`ere play Generalized Flowers. 

%Two ahead
\begin{lem}[Conway's Two-Ahead Ploy \cite{WW}]\label{lem:N2ahead} Let $F$ be a game of Flowers. If $F$ has a blue Flower but no red Flower, then Left has a winning move. Consequently, if $F$ has at least two more blue Flowers than red ones, then Left wins by cutting in the stem of any red Flower.
\end{lem}

%One more
\begin{cor}\label{cor:Nonemore}
If there is exactly one more blue Flower than the number of red Flowers and Right has a winning move, then this winning move is to cut in the stem of a blue Flower.
\end{cor}

%prop:n1,b<c
\begin{lem}[{Berlekamp, \cite{notebook}}]\label{lem:Nn1,b<c}Let $F= *_{b_1}:x_1 + *_{c_1}:-y_1 + A$ where $x_1,y_1>0$ and $A$ is a game of Stalks with nim-sum $a$. Suppose $b_1 < c_1$. Then 
\begin{eqnarray*}
o^+(F) = \left\{ 
\begin{array}{cc}
\mathcal{N} & \text{ if } a ~\text{\sout{$\uparrow$}}~  (b_1 - 1) \geq c_1\\
\mathcal{L} & \text{ if } a ~\text{\sout{$\uparrow$}}~  (b_1 - 1) < c_1.
\end{array}
\right.
\end{eqnarray*}
\end{lem}

%prop:n1,b=c
\begin{lem}\label{lem:Nn1,b=c} Let $F= *_{h}:x_1 + *_{h}:-y_1 + A$ where $x_1,y_1>0$ and $A$ is a game of Stalks with nim-sum $a$. Let $\alpha$ be the largest integer such that $2^{\alpha} | h$. Then
\begin{eqnarray*}
o^+(F) = \left\{ 
\begin{array}{cc}
\mathcal{N} & \text{ if } a\geq 2^{\alpha} \text{ or } 0 \neq a < 2^{\alpha}, x_1=y_1 \\
\mathcal{P} & \text{ if } a = 0, x_1 = y_1 \\
\mathcal{L} & \text{ if } a < 2^{\alpha}, x_1 > y_1 \\
\mathcal{R} & \text{ if } a < 2^{\alpha}, x_1 < y_1.
\end{array}
\right.
\end{eqnarray*}
\end{lem}

From these results it follows that for a game $F$ of Hackenbush Flowers under normal play, when the first player has at least one more Flower, we have $F \in \mathcal{N}^+$; when the second player has at least two more Flowers, we have $F \in \mathcal{P}^+$; and otherwise the players alternate cutting down Flowers. Moreover, from Lemmas \ref{lem:Nn1,b<c} and \ref{lem:Nn1,b=c}, once each player has chosen a Flower to cut down, the first player cuts it to maximize the nim-sum of the resulting Stalks game and the second player to minimize it. 

As these results for normal play Flowers depend only on the relative magnitudes of the $b_i$, $c_j$, $x_i$ and $y_j$, and a game of Flowers has only a finite number of Flowers, these results also extend easily to all Generalized Flowers.

%DESCRIBE IMPARTIAL GAMES, MONOIDS, FINITE MISERE QUOTIENT

%MISERE STARTS
\section{Mis\`ere Play Generalized Flowers}\label{GenFlowers}
Let $F = \sum_{i=1}^n F_i \in \mathcal{F}$ a disjunctive sum of Generalized Flowers. We note that as all Shrubs are equivalent to Stalks, and all Stalks are Generalized Flowers, classifying the outcome classes of $\mathcal{F}$ also classifies the outcomes for all disjunctive sums of Generalized Flowers and Shrubs. From here onwards we will only discuss Generalized Flowers and Stalks. Our main theorem is as follows.

\begin{thm}\label{thm:main}
Let $F \in \mathcal{F}$ be a disjunctive sum of Generalized Flowers. Then, 
\vspace{-5pt}
\begin{eqnarray*}
o^+(F)=o^-(F^*) \text{ and } o^+(F^*)=o^-(F).
\end{eqnarray*}
\end{thm}

Our proof of this theorem consists of a number of steps. In Section \ref{unequal}, we show that if Theorem \ref{thm:main} holds for Generalized Flower games with equal numbers of red and blue flowers, then it holds for all Generalized Flower games. In Section \ref{Sprigs}, we extend the results of McKay, Milley and Nowakowski to show Theorem \ref{thm:main} holds for all disjunctive sums of Generalized Sprigs and Stalks. In Section \ref{nocancel}, we show that Theorem \ref{thm:main} holds for Generalized Flower games with equal numbers of red and blue flowers where there is no pair of flowers $F_1, F_2$ such that $F_1=*_{h}:(x)$ and $F_2=*_{h}:(-x)$. In Section \ref{cancel}, we show that if Theorem \ref{thm:main} holds for all Generalized Flower games with no pair  $F_1=*_{h}:(x), F_2=*_{h}:(-x)$, then it holds for all Generalized Flower games.

The following result allows us to simplify games in $\mathcal{F}$.

%can ignore Sprigs with equal Blossoms
%assumptions about Flower games
\begin{thm}\label{thm:*equal} Let $X$ and $Y$ be games in Red-blue Hackenbush with normal-play values $x$ and $y$ respectively such that $x+y=0$. Then $*:X+*:Y \equiv^- 0$ (mod $\mathcal{D}$).
\end{thm}
\begin{proof}
Let $G$ be a game in $\mathcal{D}$ and assume without loss of generality that Left wins $G$. Then Left wins $*:X+*:Y+G$ using the following strategy. If Right cuts down one of $*:X$ or  $*:Y$, Left cuts down the other one; if Right plays in $X$ or $Y$, Left responds in either one so as to make the corresponding sum of blossom values non-negative; and otherwise Left follows her original strategy for $G$ until no move is available in $G$. We note that since $X+Y=^+0$, Left can always respond to Right's moves in $X$ or $Y$ with a move in $X$ or $Y$. 
Following this strategy, when there is no move available in $G$, it must be Left's turn. Either what remains is $0$ and Left wins, or what remains is $*:X'+*:Y'$ for some $X'+Y' \geq^+0$. In the latter case, at least one of $X' \geq^+ 0$ or $Y' \geq^+0$ must hold, so assume that $X' \geq^+0$. Then Left wins by cutting down $*:Y'$.\end{proof}

Let the \textit{blossom value} of a Generalized Flower be the normal-play value of its Blossom. Then as $*+* \equiv^- 0 \text{ (mod $\mathcal{D}$)}$, we may cancel pairs of Stalks of height $1$, and by Theorem \ref{thm:*equal}, we may cancel pairs of Generalized Sprigs with blossom values that are equal in magnitude and opposite in sign. We now introduce a result that will help us to classify Generalized Flowers in terms of their blossom values.

%thm: ordinal sums
%strategy thm
%assumptions about Flower games theorem
\begin{thm}[{McKay, Milley and Nowakowski, \cite{sprigs}}]\label{thm:ord}
Let $G$ have a left and a right option. If $H \geq^+0$ then for any game $X$, $o^-(G:H+X) \geq o^-(G+X)$ and hence $G:H \geq^- G$. Similarly, if $H \leq^-0$, then $G:H \leq^- G$.
\end{thm}

%can ignore Blossoms =^+0
%assumptions about Flower games theorem
\begin{cor}
Let $B$ be a position in red-blue Hackenbush with normal play value $0$. Then $*_n:B =^- *_n$.
\end{cor}
\begin{proof}
Since $B=^+0$, by taking $G=*_n$ and $H=B$ in Theorem \ref{thm:ord}, we have  $*_n:B \geq^- *_n$, and similarly $*_n:B \leq^- *_n$.
\end{proof}

Hence for any Generalized Flower with Blossom $B=^+0$, we may ignore the Blossom and treat the Generalized Flower as a Stalk. 
Moreover, in classifying the outcome classes of disjunctive sums of Shrubs and Generalized Flowers, the proofs and strategies outlined in Section \ref{nocancel} and Section \ref{cancel} depend only on the Shrubs and the heights and blossom values of the Generalized Flowers, yet are sufficient to determine the outcome of any disjunctive sum of Generalized Flowers. Hence we will let $*_n:(x)$ denote any Generalized Flower with height $n$ and blossom value $x$. Recall that if a Generalized Flower has blossom value $x$, we say that it is \textit{blue} if $x>0$ and \textit{red} if $x<0$. We also say that the blue Generalized Flowers belong to Left and the red Generalized Flowers to Right.

\begin{comment} %THIS MIGHT BE TRUE CHECK
%ADDED TO PREVIOUS
\underline{ADDD TO PREVIOUS}
\begin{thm}
In mis\`ere play, the canonical form of a Flower is $*_n:x$, where $x$ is the normal-play value of the red-blue part of the Flower.
\end{thm}
\end{comment}

\subsection{Generalized Flower Games with Unequal Numbers of Red and Blue Flowers}\label{unequal}

In this section, we reduce proving Theorem \ref{thm:main} for Generalized Flower games with unequal numbers of red and blue flowers to proving it for Generalized Flower games with equal numbers of red and blue flowers. We proceed by first giving the mis\`ere play equivalent of Lemma \ref{lem:N2ahead}, showing that the result holds for Generalized Flower games where a player has a two flower advantage. We then give the mis\`ere play equivalent of Corollary \ref{cor:Nonemore}, showing that games where a player has a one flower advantage can be reduced to games with equal numbers of red and blue flowers.

We first present a lemma that will aid in proving these results. 

%advantages in ordinal sum games
%'strategy' lemma
%MAY WANT SIMILAR ONE ABOUT *_n:H_1 + *_n:H_2 + X, H_1 + H_2 >^+0
\begin{lem}\label{lem:ordgreater}
Let $G$ have a left and right option and $H >^+0$. Then for any game $X$ satisfying $o^-(G+X) \in \mathcal{N}\cup \mathcal{P} \cup \mathcal{L}$, Left has a winning move under mis\`ere play in $G:H+X$. Further, for any game $Y$ satisfying $o^-(G+Y) \in \mathcal{P} \cup \mathcal{L}$, Right does not have a winning move under mis\`ere play in $G:H+Y$.
\end{lem}
\begin{proof}
If $o^-(G+X) \in \mathcal{L} \cup \mathcal{N}$ then Left has a winning move under mis\`ere play in $G+X$, so by Theorem \ref{thm:ord}, Left also has a winning move in $G:H+X$. If $o^-(G+X) \in \mathcal{P}$ then Left's winning move is to play in $H$ to a follower $H^L \geq^+0$. Then by Theorem \ref{thm:ord}, Left wins playing second in $G:H^L + X$. 

Suppose $o^-(G+Y) \in \mathcal{P} \cup \mathcal{L}$. Then for any right follower $Y^R$ of $Y$, $o^-(G+Y^R) \in \mathcal{N} \cup \mathcal{L}$ so that Left has a winning move in $G:H+Y^R$ and Right cannot win in $G:H+Y$ by first playing in $Y$. For any right follower $G^R$ of $G$, $o^-(G^R+Y) \in \mathcal{N} \cup \mathcal{L}$ so that Right cannot win in $G:H+Y$ by first playing in $G$. Finally, since $H \geq^+0$, any right follower $H^R$ satisfies $H^R >^+0$, so that Left has a winning move in $G:H^R+Y$ and Right cannot win in $G:H+Y$ by first playing in $H$. 
\end{proof}

We present now the mis\`ere play versions of Lemma \ref{lem:N2ahead} and Corollary \ref{cor:Nonemore}.
%Two ahead
\begin{lem}\label{lem:2ahead} Let $F \in \mathcal{F}$ be a game. If $F$ has a blue Generalized Flower but no red Generalized Flower, then Left has a winning move under mis\`ere play. Consequently, if $F$ has at least two more blue Generalized Flowers than red ones, then Left wins under mis\`ere play, and in her next move she may win by cutting in the stem of any red Generalized Flower.
\end{lem}
\begin{proof}
If $F$ has a blue Generalized Flower but no red Generalized Flower, then $F$ is of the form $*_{h_1}:B_1+*_{h_2}:B_2 + \cdots + *_{h_n}:B_n + F'$, where $F'$ is a stalks game, $*_{h_1} + \cdots + *_{h_n} + F' \in \mathcal{N}^- \cup \mathcal{P}^-$, and $B_i >^+0$ for all $i$. Hence by using Lemma \ref{lem:ordgreater} and using induction on $n$, it is easy to see that Left has a winning move.
If $F$ has at least two more blue Generalized Flowers than red ones, Left wins by continually cutting down red Generalized Flowers to Stalks of any height. Eventually, there will be no red Generalized Flowers, at least one blue Generalized Flower and it will be Left's turn, and so Left wins.
\end{proof}

%One more
\begin{cor}\label{cor:onemore}
If there is exactly one more blue Generalized Flower than the number of red Generalized Flowers and Right has a winning move under mis\`ere play, then this winning move is to cut in the stem of a blue Generalized Flower.
\end{cor}
\begin{proof}
Suppose Right does not cut down a blue Generalized Flower. Then Left cuts down a red Generalized Flower, leaving a game with at least two more blue Generalized Flowers than red ones, so by Lemma \ref{lem:2ahead}, Left wins.
\end{proof}

By Lemma \ref{lem:2ahead} and Corollary \ref{cor:onemore}, we see that to find the outcome of any game in $\mathcal{F}$, it suffices to understand game positions with $n$ blue and $n$ red Generalized Flowers, where $n \in \mathbb{N^+}$.

%SECTION	RESULTS FOR SPRIGS
%PROVE FOR SPRIGS
\subsection{Mis\`ere Play Generalized Sprigs and Stalks}\label{Sprigs}	
	
The mis\`ere outcome classes for sums of Sprigs were completely determined by McKay, Milley and Nowakowski in \cite{sprigs}. They showed that $o^-(G) = o^+(G+*)$ and $o^+(G)=o^-(G+*)$ for all such games. This section uses similar methods to show analogous results for disjunctive sums of Generalized Sprigs and extends them to games of Generalized Sprigs where the Stalks can have height greater than $1$.

Let $G= \sum_{i=1}^m *:(x_i) + \sum_{j=1}^n *:(-y_j) + A$ be a sum of Generalized Sprigs and Stalks. Let $x_i, y_j >0$ and let $A$ be a game of Stalks with nim-sum $a$. We note that by Theorem \ref{thm:*equal}, we may assume that there is no pair $x_i$ and $y_j$ such that $x_i=-y_j$. 

Following the notation of \cite{sprigs}, we define the \textit{advantage} of $G$ to be $\Delta(G)=m-n$. We define the \textit{edge} of $G$ to be $\epsilon(G) = \min\{x_i\}-\min\{y_j\}$, where if either $\{x_i\}$ or $\{y_j\}$ is empty, we take $\epsilon(G)=0$. 

\begin{thm}\label{thm:Generalized Sprigs}
Let $G = \sum_{i=1}^m *:(x_i) + \sum_{j=1}^n *:(-y_j) + A$ be a disjunctive sum of Sprigs and Stalks, where $0<x_1 \leq \cdots \leq x_m, 0<y_1 \leq \cdots \leq y_n$ and $A$ is a game of Stalks with nim-sum $a$.
\begin{itemize}
%(i)
\item[(i)] If $|\Delta(G)| \geq 2$, then
\vspace{-10pt}\begin{eqnarray*}
o^+(G) = o^-(G^*)= \left\{
\begin{array}{cc}
\mathcal{L} & \text{ if } \Delta(G) \geq 2 \\
\mathcal{R} & \text{ if } \Delta(G) \leq -2 .
\end{array}
\right.
\end{eqnarray*}
%(ii) FIX INCLUDE epsilon=0
\item[(ii)]If $|\Delta(G)| = 1$, then
\begin{eqnarray*}
o^+(G) = o^-(G^*)= \left\{
\begin{array}{cc}
\mathcal{N} & \text{ if } \Delta(G)=1 \text{ and } \epsilon(G) \leq 0 \text{ and } a = 0 \\
& \text{ or } \Delta(G)=-1 \text{ and } \epsilon(G) \geq 0 \text{ and } a=0\\
\mathcal{L} & \text{ if } \Delta(G)=1 \text{ and } \epsilon(G)>0 \text{ and } a=0 \\
& \text{ or } \Delta(G)=1 \text{ and } a \neq 0\\
\mathcal{R} & \text{ if } \Delta(G)=-1 \text{ and } \epsilon(G)<0 \text{ and } a=0 \\
& \text{ or } \Delta(G)=-1 \text{ and } a \neq 0 .
\end{array}
\right.
\end{eqnarray*}
%(iii)
\item[(iii)]If $|\Delta(G)| = 0$, then
\begin{eqnarray*}
o^+(G) = o^-(G^*) =\left\{
\begin{array}{cc}
\mathcal{N} & \text{ if } a \neq 0\\
\mathcal{P} & \text{ if } \epsilon(G) = 0 \text{ and } a = 0\\
\mathcal{L} & \text{ if } \epsilon(G)>0 \text{ and } a=0\\
\mathcal{R} & \text{ if } \epsilon(G)<0 \text{ and } a=0 .
\end{array}
\right.
\end{eqnarray*}
\end{itemize}
\end{thm}
\begin{proof}
Statement (i) is true by Lemma \ref{lem:2ahead}. We prove statements (ii) and (iii) concurrently by induction on $m+n$.

\noindent\textit{Base cases: $m+n=0$, $m+n=1$ and $m+n=2$}

If $m+n=0$, then $m=n=0$, $\epsilon(G)=0$ and the results for $o^+(G)$ and $o^-(G^*)$ follow from the results for mis\`ere play nim. 

If $m=1$, $n=0$, the game $G^*$ is of the form $*:B+A^*$, where $*$ is the stem of the blue Generalized Flower, $B$ is its blossom and $A^*$ is the remaining Stalks game. As $*+A^*$ is an impartial game, $o^-(*+A^*) \in \mathcal{N} \cup \mathcal{P}$, and so by taking $G=*$, $H=B$ and $X=A^*$ in Lemma \ref{lem:ordgreater}, Left has a winning move under mis\`ere play. Further, if Right has a winning move, it must be to cut down the blue flower, so Right has a winning move if and only if $A^* \in \mathcal{P}^-$, which, by the results for nim in Lemma \ref{lem:nim}, occurs if and only if $a=0$. The case $m=0$, $n=1$ is similar.

If $m=1$ and $n=1$, then $\epsilon(G) \neq 0$, as we have assumed that no pair of sprigs cancels. Suppose $\epsilon(G)>0$. If $a \neq 0$, the next player wins by cutting the stem of the opponent's flower. If $a = 0$, Left can win using the following strategy. She plays only in the blossoms of the two flowers, until Right cuts in the stem of a flower, or a stalk. If Right cuts in a stalk, then Left moves to a follower $G'$ of $G^*$, where ${(G')}^*$ has nim-sum $a' \neq 0$ and so Left wins playing next. If Right cuts in a stem, then Left cuts in the other stem, leaving a stalks game $G''$ where $(G'')^*$ has nim-sum $0$, and so the previous player (Left) wins. The case $\epsilon(G)<0$ is similar.\vspace{10pt}\\
\noindent\textit{Inductive Step}

We note first that since $m+n\geq 3$, $\epsilon(G) \neq 0$, so suppose without loss of generality that $\epsilon(G)>0$. 

If $|\Delta(G)|=0$, Left wins playing next by cutting in the stem of the red Generalized Flower with blossom value $y_n$, moving to a follower $G'$ of $G^*$ with $\Delta((G')^*)=1$ and $\epsilon((G')^*) > 0$, which Left wins by induction. Further, if Right has a winning move, it must be to cut down a blue flower, moving to a game $G''$ with $\Delta((G'')^*)=-1$ and $\epsilon((G')^*) \geq 0$, so by induction Right has a winning move if and only if $a \neq 0$. 

If $\Delta(G) = 1$ and Left plays first, she cuts down a red flower and wins by the two-ahead ploy. Further, if Right has a winning move, it must be to cut down a blue flower, moving to a follower $G'$ of $G^*$ with $|\Delta((G')^*)|=0$ and $\epsilon((G')^*) \geq \epsilon(G) > 0$, so by induction Right cannot win.

If $\Delta(G) = -1$ and Right plays first, he cuts down a blue flower and wins by the two-ahead ploy. Further, if Left has a winning move, it must be to cut down a red flower, moving to a game $G''$ with $|\Delta((G'')^*)|=0$ and $\epsilon((G'')^*) \geq \epsilon(G) > 0$, so by induction Left has a winning move if and only if $a = 0$.

%L next: L wins
%R next: if R doesn't cut down L's Generalized Flower, L wins
	%Hence R needs to win m=0, n=0, so need to win a
	%so a previous player win
	%so a=0 or a just 1 Stalk
%Stalk height \geq 2
%Stalk height 1

%maybe instead of min, take min

\end{proof}

%SECTION	NO EQUAL
\subsection{Generalized Flowerbeds with no Canceling Generalized Flowers}\label{nocancel}
Let $F \in \mathcal{F}$ be a game position with $n$ blue Generalized Flowers, $n$ red Generalized Flowers, and Stalks $A$ with nim-sum $a$. We call such a game a \textit{Generalized Flowerbed}.

Let the blue Generalized Flowers have heights $b_1 \geq b_2 \geq \cdots \geq b_n$ and positive Blossom normal-play values $\{x_i\}$, where if $b_i=b_{i+1}$ then $x_i \leq x_{i+1}$. Similarly define red Generalized Flower heights $c_1 \geq c_2 \geq \cdots \geq c_n$ and negative blossom values $\{-y_i\}$. 

%defn of strong and weak Generalized Flowers
Given two Generalized Flowers $F_1 = *_{h_1}:(x_1)$ and $F_2 = *_{h_2}:(x_2)$, we say that $F_1$ is \textit{weaker} than $F_2$ and $F_2$ is \textit{stronger} than $F_1$ if $h_1>h_2$, or $h_1=h_2$ and $|x_1|<|x_2|$. We say that $F_1$ and $F_2$ are \textit{equally weak} if $h_1=h_2$ and $x_1=x_2$. If a red Flower and a blue Flower are equally weak, we note that they form a zero-sum game under normal play and say that they \textit{cancel}. In Section \ref{cancel}, we show that, given a Generalized Flowerbed $F$, we may either remove all pairs of canceling flowers to obtain a game $F_t$ such that $o^-(F) = o^+(F_t)$, or replace a maximal set of canceling flowers with the game $*$ to obtain a simpler game $F'$ such that $o^-(F) = o^-(F')$. In this section, assume that we have a Generalized Flowerbed with a Generalized Flower of height at least $2$ such that no pair of Generalized Flowers cancels, and prove Theorem \ref{thm:main} for such game positions. Note that in these cases, $F^*=F$.

%	SECTION		2 FLOWERS
\subsubsection{Generalized Flowerbeds of Size $1$}

Let $G$ be a Generalized Flowerbed with $n=1$, a Generalized Flower of height at least $2$, and no pairs of Generalized Flowers that cancel. The normal play outcomes of such a game were presented in Lemma \ref{lem:Nn1,b<c} and Lemma \ref{lem:Nn1,b=c}. We find the mis\`ere play outcomes of these games and show that they are identical to the outcomes under normal play. This, together with the results for Generalized Sprigs in Section \ref{Sprigs}, proves Theorem \ref{thm:main} for Generalized Flowerbeds of size $1$.

%prop:n1,b<c
\begin{prop}\label{prop:n1,b<c}Let $F \in \mathcal{F}$ be a position with $n=1$ and $b_1 < c_1$. Then 
\begin{eqnarray*}
o^+(F) = o^-(F) = \left\{ 
\begin{array}{cc}
\mathcal{N} & \text{ if } a ~\text{\sout{$\uparrow$}}~  (b_1 - 1) \geq c_1\\
\mathcal{L} & \text{ if } a ~\text{\sout{$\uparrow$}}~  (b_1 - 1) < c_1 .
\end{array}
\right.
\end{eqnarray*}
\end{prop}
\begin{proof}
We show first that if Left moves first, she wins under mis\`ere play. 

Suppose that Left moves first and that there is a Stalk of height at least $2$. We show that Left can win by cutting the red Generalized Flower to a Stalk of height $c_1'<c_1$ such that $a ~\text{\sout{$\uparrow$}}~ (c_1-1) = a \oplus c_1'$. 
If Left does so, by Corollary \ref{cor:onemore}, Right must then cut the blue Generalized Flower to a Stalk of height $b_1' < b_1$. This leaves Stalks with nim-sum $b_1' \oplus (a \oplus c_1')$, where 
\begin{eqnarray*}
a \oplus c_1' = a ~\text{\sout{$\uparrow$}}~ (c_1-1) \geq c_1-1 \geq b_1 > b_1'.
\end{eqnarray*}
Since $a \oplus c_1' > b_1'$, their nim-sum is not 0, and so because there is a Stalk of height at least $2$, Left plays next and wins. 

Suppose that Left moves first and all the Stalks have height $1$. As pairs of Stalks of height $1$ can be cancelled, we then have the following cases:
\begin{itemize}
\item \textit{Case 1: $b_1=1$}\\
Left cuts the red Generalized Flower to a Stalk of height $a=0$ or $1$. Then by Corollary \ref{cor:onemore}, Right must cut the stem of the blue Generalized Flower, leaving an even number of Stalks of height 1. Hence Left wins.
\item \textit{Case 2a: $b_1 >1$, no Stalks}\\
Left cuts the red Generalized Flower to a Stalk of height $b_1$, leaving a game with nim-sum 0 and at least one Stalk of height at least $2$. Hence Left wins. 
\item \textit{Case 2b: $b_1>1$, one Stalk}\\
Left cuts the Stalk. If Right does not cut the blue Generalized Flower to a Stalk, we return to Case 2a and Left wins. Hence Right must cut the blue Generalized Flower to a Stalk of height $b_1'<b_1$, and Left cuts the red Generalized Flower to a Stalk of height $0$, $1$ or $b_1'$, at least one of which will be a winning move.
\end{itemize}
Suppose that Right moves first. If he does not cut down the blue Generalized Flower, we move either to a game with $n=1$ and $b_1<c_1$ with Left first, or to a game with blue Generalized Flowers but no red ones, so Left wins. Hence Right cuts the blue Generalized Flower to a Stalk of height $b_1'<b_1$.

If $a ~\text{\sout{$\uparrow$}}~  (b_1 - 1) \geq c_1$, Right takes $b_1'$ so that $a \oplus b_1' = a ~\text{\sout{$\uparrow$}}~  (b_1 - 1) \geq c_1$. Left then must cut the red Generalized Flower to a Stalk of height $c_1' < a \oplus b_1'$, leaving a Stalks game with non-zero nim-sum. Since $ a ~\text{\sout{$\uparrow$}}~  (b_1 - 1) \geq c_1>b_1 \geq 1$, there must be a Stalk of height at least $2$ and Right wins.

If $a ~\text{\sout{$\uparrow$}}~  (b_1 - 1) < c_1$, then Left can cut the red Generalized Flower to a height of $0$, $1$, or $a \oplus b_1'$, at least one of which is a winning position.
\end{proof}

%prop:n1,b=c
\begin{prop}\label{prop:n1,b=c} Let $F \in \mathcal{F}$ be a position with $n=1$ and $b_1=c_1 \geq 2$. Let $\alpha$ be the largest integer such that $2^{\alpha} | b_1$. Then
\vspace{-10pt}
\begin{eqnarray*}
o^+(F) = o^-(F) = \left\{ 
\begin{array}{cc}
\mathcal{N} & \text{ if } a\geq 2^{\alpha} \text{ or } 0 \neq a < 2^{\alpha}, x_1=y_1 \\
\mathcal{P} & \text{ if } a = 0, x_1 = y_1 \\
\mathcal{L} & \text{ if } a < 2^{\alpha}, x_1 > y_1 \\
\mathcal{R} & \text{ if } a < 2^{\alpha}, x_1 < y_1.
\end{array}
\right.
\end{eqnarray*}
\end{prop}
	%proof of prop:n1,b=c
\begin{proof}Let $b_1 = c_1 = 2^{\alpha}+2^{\alpha+1}d$ for some $d \in \mathbb{N}$.
\begin{itemize}
\item \textit{Case 1: $a \geq 2^{\alpha}$}\\
Suppose that Left plays first. Then she cuts the red Generalized Flower to a Stalk of height $c_1' < c_1$ such that $a ~\text{\sout{$\uparrow$}}~ (c_1-1) = a \oplus c_1'$. Since $a \geq 2^{\alpha}$ and $c_1 = 2^{\alpha}+2^{\alpha+1}d$, we have
\begin{eqnarray*}
a \oplus c_1' = a ~\text{\sout{$\uparrow$}}~ (c_1-1) \geq 2^{\alpha+1} ~\text{\sout{$\uparrow$}}~ \left(\sum_{i=0}^{\alpha-1} 2^{i} + 2^{\alpha+1}d\right) = 
\sum_{i=0}^{\alpha}2^i + 2^{\alpha+1}d  >c_1 .
\end{eqnarray*}
Then by Corollary \ref{cor:onemore}, Right must cut the blue Generalized Flower to a Stalk of height $b_1' < c_1$. This leaves Stalks with nim-sum $b_1' \oplus (a \oplus c_1')$, where $a \oplus c_1' > b_1'$, so that their nim-sum is not 0 and there is at least one Stalk of height 2. Hence Left wins under both normal and mis\`ere play. Similarly, if Right plays first, Right wins.
\item \textit{Case 2: $a < 2^{\alpha}$ and $x_1>y_1$, or $a < 2^{\alpha}$ and $x_1<y_1$}\\
In the first subcase, as $a ~\text{\sout{$\uparrow$}}~ (b_1-1) =a ~\text{\sout{$\uparrow$}}~ (c_1-1) = b_1-1$, if Left cuts $c_1$ to a Stalk of height $c_1'<c$, then Right can cut $b_1$ to a Stalk of height $0$, $1$ or $a \oplus c_1' < b_1$, at least one of which is a winning move. Similarly, if Right cuts $b_1$ then Left has a winning move. Hence the first player will always cut an edge in the Stalks or in a Blossom. Moreover, if a player cuts a Stalk so as to increase the Stalk nim-sum to $a' \geq 2^{\alpha}$ then by Case 1 the other player wins.

Left's winning strategy is then as follows: cut edges in the Stalks until either there are no more Stalks or removing any edge in a Stalk would increase the Stalk nim-sum to at least $2^{\alpha}$. Then they are in some position with Stalk nim-sum $a' < 2^{\alpha}$ and $x_1>y_1$ blossom values, where neither player can play in the Stalks or stems. Then Right is forced to play only in the Blossoms, and Left is able to retain a game with positive normal play value in the Blossoms, so that in particular there is always a blue Generalized Flower. Eventually one Generalized Flower is blue, the other a Stalk and it is Left's turn, and hence Left wins. Similarly in the other subcase Right wins.
\item \textit{Case 3: $a < 2^{\alpha}$, $x_1=y_1$ and $a = 0$}\\
Let Left be the second player. Then Left wins with the following strategy. If Right cuts the blue Generalized Flower to a Stalk of height $h$, Left cuts the red Generalized Flower to a Stalk of height $0$, $1$ or $h$, at least one of which will be a winning move. If Right  cuts in a Blossom, the new blossom values $x_1'$ and $-y_1'$ satisfy $x_1'>y_1'$ so that by Case 2, Left wins. If Right cuts a Stalk, then Left is able to cut a Stalk to keep the Stalk nim-sum $0$. Eventually Right must cut a Generalized Flower Stalk or Blossom and Left wins. Similarly, if Right plays second, he wins.
\item \textit{Case 4: $a < 2^{\alpha}$, $x_1=y_1$ and $a \neq 0$}\\
Then the first player cuts a Stalk so that the new Stalk nim-sum is $0$, and by Case 3 they win.
\end{itemize}
\end{proof}

%SECTION	n >= 2
\subsubsection{Generalized Flowerbeds of Size $n \geq 2$}

In this section, we prove Theorem \ref{thm:main} for Generalized Flowerbeds of size $n \geq 2$ containing a Generalized Flower of height at least $2$, such that no pair of Generalized Flowers cancels. We proceed by first showing that the theorem is true when the player with the weakest Generalized Flower plays second.

%thm:1win
\begin{thm}\label{thm:1win}
Let $F \in \mathcal{F}$ be a Generalized Flowerbed such that there is a Generalized Flower of height at least $2$, there are no pairs of canceling Flowers, and the weakest Generalized Flower is red. Then, under both play conventions, Left has a winning move. Moreover, if $n \geq 2$, then Left has a move that is winning under both normal and mis\`ere play.
\end{thm}
\begin{proof}   
We note that if the weakest blue Flower is stronger than the weakest red Flower, then either $b_1<c_1$ or $b_1=c_1 \geq 2$ and $x_1>y_1$. Thus when $n=1$, by Propositions \ref{prop:n1,b<c} and \ref{prop:n1,b=c}, Left has a winning move.

For general $n$, Left simply attacks all the red Generalized Flowers, leaving the weakest one untouched. By Corollary \ref{cor:onemore}, Right must reciprocate, and the players eventually reach the case $n=1$ where the weakest red Flower is still weaker than all the blue Flowers, and Left wins. Thus Left wins under both normal and mis\`ere play by first cutting down any red Generalized Flower that is not the weakest one.
\end{proof}

\begin{cor}\label{cor:2play}
Let $F \in \mathcal{F}$ be a Generalized Flowerbed such that there is a Generalized Flower of height at least $2$, there are no pairs of canceling Flowers, and the weakest Generalized Flower is red. Then if Right has a winning move, it must be to cut down a blue Flower.
\end{cor}
\begin{proof}
Suppose Right does not cut down a blue Flower. Then the resulting game is either a Flowerbed with weakest Generalized Flower red, and Left wins, or a Flowerbed with more blue Flowers than red Flowers, and Left wins by cutting down a red Flower.
\end{proof}

It remains to analyze the outcome when the player with the weakest Generalized Flower plays first. We proceed as follows. In Lemma \ref{lem:1w,>2}, Corollary \ref{cor:1w,>2}, we first restrict our attention to a class of Generalized Flowerbeds in which the winning player can win without reaching a game of Generalized Sprigs, and show that the winning player follows identical strategies under both play conventions until at most one flower is left. In Lemma  \ref{lem:1ww} and Proposition \ref{prop:1first}, we then show that all Generalized Flowerbeds with a Generalized Flower of height at least $2$ fall in this class. These results prove Theorem \ref{thm:main} for Generalized Flowerbeds without a pair of canceling Generalized Flowers.

%%%%BREAKING UP LEMMA 6.7

%first player has weakest, winner can guarantee >= 2
%Flowerbeds strategy
\begin{lem}\label{lem:1w,>2}
Let $F \in \mathcal{F}$ be a Generalized Flowerbed with size $n \geq 2$ such that there is a Generalized Flower of height at least $2$ and there are no pairs of canceling Flowers. Suppose that either player has a winning strategy such that at each turn, there is either at least one Flower of height at least $2$, or at most one Flower. Then an optimal strategy for each player under both normal and mis\`ere play is to play as follows until the game is a Generalized Flowerbed of size $1$:
\begin{itemize}
\item Choose one of the opponent's Generalized Flowers, with the restriction that if the first player has the weakest Generalized Flower, then that Generalized Flower cannot be chosen. 
\item Cut it down so as to either maximize (first player) or minimize (second player) the nim-sum of the Stalks in the resulting game.
\end{itemize}
The best choice of Flower is the same under both play conventions, but there is no known method of efficiently determining which one it is.
\end{lem} 

\begin{proof}
We note that this is vacuously true when the second player has the weakest Flower, so assume that the first player has the weakest Flower.
We assume that Left plays first and prove the statement by induction on $n$. When $n \geq 2$, by Corollary \ref{cor:2play}, Left must first cut down a red Flower, and will still have the weakest Flower. Following this, by Corollary \ref{cor:onemore}, Right must cut down a blue Flower, and by Theorem \ref{thm:1win} must make sure that Left will still have the weakest Flower. Hence by induction, the players take turns cutting down Flowers until the game is a Generalized Flowerbed of size $1$.

Following this strategy of play, the game eventually reaches a Generalized Flowerbed of size $1$ with a blue Flower $*_{b_{\sigma(n)}}:(x_{\sigma(n)})$, a red Flower $*_{c_{\sigma(n)}}:(-y_{\sigma(n)})$ and Stalks with nim-sum $a' = a \oplus b'_{\sigma(1)} \oplus c'_{\tau(1)} \oplus \cdots \oplus b'_{\sigma(n-1)} \oplus c'_{\tau(n-1)}$, for $0 \leq b_i' < b_i$ and $0 \leq c_j' < c_j$ for all $i, j$, where $\sigma$ and $\tau$ give the order in which blue and red Generalized Flowers are cut down . If $b_{\sigma(n)}>c_{\tau(n)}$, by Proposition \ref{prop:n1,b<c}, Left wins if and only if $a' ~\text{\sout{$\uparrow$}}~ (c_{\tau(n)}-1) \geq b_{\sigma(n)}$. If $b_{\sigma(n)}=c_{\tau(n)}$ and $x_{\sigma(n)}<y_{\tau(n)}$, by Proposition \ref{prop:n1,b=c}, Left wins if and only if $a' \geq 2^{\alpha+1}$. Hence, in both cases, Left wants to maximize $a'$ and Right to minimize it. Since \sout{$\uparrow$} and \sout{$\downarrow$} are both increasing in the first variable, Left cuts every red Flower so as to maximize the nim-sum of the Stalks, and Right cuts every blue Flower so as to minimize it, as required. 
\end{proof}

%conditions for winning if first player has weakest and winning can guar >= 2
\begin{cor}\label{cor:1w,>2}
Let $F \in \mathcal{F}$ be a Generalized Flowerbed with size $n \geq 2$ such that there is a Generalized Flower of height at least $2$ and there are no pairs of canceling Flowers.
Suppose that a player has a winning strategy where they can guarantee that at each turn, there is either at least one Flower of height at least $2$, or at most one Flower. Suppose also that Left plays first, and that the Generalized Flowers are cut down in the order $(b_{\sigma(1)}, \ldots, b_{\sigma(n)})$ and $(c_{\tau(1)}, \ldots, c_{\tau(n)})$ for some permutations $\sigma$ and $\tau$. Then Left wins if and only if 
\begin{eqnarray*}
&a ~\text{\sout{$\uparrow$}}~ (c_{\tau(1)}-1) ~\text{\sout{$\downarrow$}}~ (b_{\sigma(1)}-1) ~\text{\sout{$\uparrow$}}~ \cdots ~\text{\sout{$\downarrow$}}~ (b_{\sigma(n-1)}-1) ~\text{\sout{$\uparrow$}}~ (c_{\tau(n)}-1)
\geq b_{\sigma(n)} ~~ \text{if } b_{\sigma(n)} >c_{\tau(n)} \\
&a ~\text{\sout{$\uparrow$}}~ (c_{\tau(1)}-1) ~\text{\sout{$\downarrow$}}~ (b_{\sigma(1)}-1) ~\text{\sout{$\uparrow$}}~ \cdots ~\text{\sout{$\uparrow$}}~ (c_{\tau(n-1)}-1) ~\text{\sout{$\downarrow$}}~ (b_{\sigma(n-1)}-1)
 \geq 2^{\alpha} ~~~~~~~~~~~~~~~~~~~~\\
 &~~~~~~~~~~~~~~~~~~~~~~~~~~~~~~~~~~~~~~~~~~~~~~~~~~~~~~~~~~~~~~~~~~~~~~~~\text{if } b_{\sigma(n)} =c_{\tau(n)}, ~x_{\sigma(n)} < y_{\tau(n)} 
\end{eqnarray*} 
where the order of operations is from left to right, and $\alpha$ is the largest integer such that $2^{\alpha} | b_{\sigma(n)}$.
\end{cor} 
\begin{proof}
This follows from Lemma \ref{lem:1w,>2} and the definitions of \sout{$\uparrow$} and \sout{$\downarrow$}.
\end{proof}

%Theorem
\begin{lem}\label{lem:1ww}
Let $F \in \mathcal{F}$ be a Generalized Flowerbed of size $n$ such that no pair of Flowers cancels. Suppose that the first player has the weakest Generalized Flower. Suppose also that the losing player has a Generalized Flower of height at least $2$. Then the winning player has a winning strategy such that at each turn, there is either at least one Flower of height at least $2$, or at most one Flower.
\end{lem}

Before proving this lemma, we present some technical lemmas about the upper and lower nim-sum that will be used in the proof.

%lem for mixUD from B
\begin{lem}[Berlekamp, \cite{notebook}]\label{lem:mixUD}
For any $a, b, c$ non-negative integers,
\begin{eqnarray*}
(a ~\text{\sout{$\uparrow$}}~ b) ~\text{\sout{$\downarrow$}}~ c \leq (a ~\text{\sout{$\downarrow$}}~ c) ~\text{\sout{$\uparrow$}}~ b.
\end{eqnarray*}
\end{lem}

%lem for consecutive downs %CHECK WHEN THIS IS USEd
\begin{lem}\label{lem:down}
Let $a, b_1, \ldots, b_n$ be non-negative integers. Then $a ~\text{\sout{$\downarrow$}}~ b_{\sigma(1)}  ~\text{\sout{$\downarrow$}}~ \cdots ~\text{\sout{$\downarrow$}}~ b_{\sigma(n)} >0$ for some permutation $\sigma$ if and only if $a ~\text{\sout{$\downarrow$}}~ b_{\tau(1)}  ~\text{\sout{$\downarrow$}}~ \cdots ~\text{\sout{$\downarrow$}}~ b_{\tau(n)} >0$ for every permutation $\tau$. 
\end{lem}
\begin{proof}
Let $(a)_n$ denote the $n$th digit in the binary representation of $a$. Let $\alpha$ be the largest integer such that $(a)_{\alpha}=1$ and $(b_i)_{\alpha}=0$ for all $i$. Let $\beta$ be the largest integer such that if $(a)_{\beta}=0$, there is at least one $i$ such that $(b_i)_{\beta}=1$, and if $(a)_{\beta}=1$, there are at least two $i$ such that $(b_i)_{\beta}=1$. 

It follows from Lemma \ref{lem:updown} that $a ~\text{\sout{$\downarrow$}}~ b_{\sigma(1)}  ~\text{\sout{$\downarrow$}}~ \cdots ~\text{\sout{$\downarrow$}}~ b_{\sigma(n)} >0$ if and only if $\alpha > \beta$, which is independent of the order of the $b_i$. We omit the calculations here.
\end{proof}

\begin{proof}[Proof of Lemma \ref{lem:1ww}] We assume without loss of generality that Left moves first and prove the result separately for Left winning and Right winning by induction on $n$.

%proof of i
Suppose Left wins. If $n=1$, the result is trivial. For $n\geq 2$, consider the following cases.
\begin{itemize}
\item\textit{Case 1: There exists at least two red Generalized Flowers of height at least $2$}\\
Then by Lemma \ref{lem:1w,>2}, after both players move, we have $n-1$ Generalized Flowers of each color, Left has the weaker weakest Generalized Flower, and there is a red Generalized Flower height at least $2$, so by induction we are done.
\item\textit{Case 2: There is exactly one red Generalized Flower of height at least $2$, i.e. for some $c_1 \geq 2$,}
\begin{eqnarray*}
F = \sum_{i=1}^n *_{b_i}:(x_i) + \sum_{j=2}^n *:(-y_j) + *_{c_1}:(y_1)+ a.
\end{eqnarray*}
Suppose for the sake of contradiction that Left can only win if she cuts the stem of the red Generalized Flower with height at least $2$. This means that there exists some $c_1' < c_1$ such that for all $k$ and $b_k'<b_k$
\begin{eqnarray*}
F_k = \sum_{i \neq k}*_{b_i}:(x_i) + \sum_{j=2}^n *:(y_j) + a \oplus c_1' \oplus b_k' ~~~\in ~~~\mathcal{N}^- \cup \mathcal{L}^-.
\end{eqnarray*}
In particular, Left wins if Right plays so as to leave a Generalized Flower of height $b_i \geq 2$ until last. In such a situation, by Lemma \ref{lem:1w,>2} the players take turns cutting Generalized Flowers until they reach some position $F'$ with one Generalized Flower of each color:
\begin{eqnarray*}
F' = *_{b_r}:(x_r) + *:(-y_s)+a'.
\end{eqnarray*}
Since Left wins, by Proposition \ref{prop:n1,b<c} we have $a' \geq b_r$, or equivalently $a' ~\text{\sout{$\downarrow$}}~ (b_r-1) > 0$. Thus by Corollary \ref{cor:1w,>2}, for any permutation $\sigma$ with $b_{\sigma(n)}\geq 2$, we have
\begin{eqnarray}\label{eqn:down}
0<a' ~\text{\sout{$\downarrow$}}~ (b_r-1) = a ~\text{\sout{$\uparrow$}}~ (c_1-1) ~\text{\sout{$\downarrow$}}~ (b_{\sigma(1)}-1) ~\text{\sout{$\downarrow$}}~ \cdots ~\text{\sout{$\downarrow$}}~ (b_{\sigma(n)}-1).
\end{eqnarray}
Moreover, by assumption Left loses if she cuts down a red Generalized Flower with height $1$ first, so in particular if she leaves the weakest red Generalized Flower until last, she loses. This means that there exists a permutation $\sigma$ with $b_{\sigma(n)} \geq c_1 \geq 2$, such that
\begin{eqnarray*}
a ~\text{\sout{$\downarrow$}}~ (b_{\sigma(1)}-1) ~\text{\sout{$\downarrow$}}~ (b_{\sigma(2)}-1) ~\text{\sout{$\downarrow$}}~ \cdots ~\text{\sout{$\downarrow$}}~ (b_{\sigma(n-1)}-1) ~\text{\sout{$\uparrow$}}~ (c_1-1)
 \leq b_{\sigma(n)}-1 \\
 %%%
\text{or } ~ a ~\text{\sout{$\downarrow$}}~ (b_{\sigma(1)}-1) ~\text{\sout{$\downarrow$}}~ (b_{\sigma(2)}-1) ~\text{\sout{$\downarrow$}}~ \cdots ~\text{\sout{$\downarrow$}}~ (b_{\sigma(n-1)}-1) ~\text{\sout{$\uparrow$}}~ (c_1-1)
\leq 2^{\alpha+1}-1 \leq b_{\sigma(n)}-1
\end{eqnarray*}
so in both cases we have
\begin{eqnarray}
 a ~\text{\sout{$\downarrow$}}~ (b_{\sigma(1)}-1) ~\text{\sout{$\downarrow$}}~ \cdots ~\text{\sout{$\downarrow$}}~ (b_{\sigma(n-1)}-1) ~\text{\sout{$\uparrow$}}~ (c_1-1) ~\text{\sout{$\downarrow$}}~ (b_{\sigma(n-1)}-1)= 0 ,
\end{eqnarray}
which, by repeated application of Lemma \ref{lem:mixUD}, contradicts (\ref{eqn:down}). Hence the original assumption was false and Left can win by cutting the stem of a red Generalized Flower with height $1$, and by induction we are done.

\end{itemize}

Suppose Right wins. If $n=1$, the result is trivial. For $n \geq 2$, consider the following cases.
\begin{itemize}
\item \textit{Case 1: Left leaves a red Generalized Flower of height at least $2$}\\
Then as Right must leave a weaker blue Generalized Flower, he also leaves a blue Generalized Flower of height at least $2$ and by induction we are done.

\item \textit{Case 2: Left leaves red Generalized Flowers all of height $1$ and Stalks nim-sum $a$}\\
If Right leaves a Generalized Flower of height at least $2$, by induction we are done. So suppose that Right cannot leave a Generalized Flower of height at least $2$. Since originally there was a Generalized Flower of height at least $2$, and the weakest Generalized Flower is blue, this means that this is the only blue Generalized Flower with height $b \geq 2$ and Right must cut it to win.

\indent Hence, for some $b' < b$, second player Right wins a Generalized Sprigs game with $\Delta(G)=0$, $\epsilon(G) \neq 0$, and Stalk nim-sum $a \oplus b'$. So, by Theorem \ref{thm:Generalized Sprigs} we must have $a \oplus b' = 0$ and some Stalk of height at least $2$, or $a \oplus b' = 1$ and there is exactly one Stalk of height exactly $1$. 

We now show in this case that Right could have won by leaving the blue Generalized Flower of height $b$. His winning strategy is to cut down all the blue Generalized Sprigs first. Since the weakest Flower is blue, the players alternate cutting down Flowers until there is one blue Flower, height $b$, some Stalks, and it is Right's turn. Moreover, since all the Flowers that were cut down were Sprigs, the Stalk configuration is the same as after Left's first move. But we have some $b'$ for which $a \oplus b' = 0$ if there is some Stalk of height at least $2$, or $a \oplus b' = 1$ if there is exactly one Stalk of height exactly $1$, so Right can win by cutting the blue Flower to a Stalk of height $b'$. Hence Right can always leave a blue Generalized Flower of height at least $2$ and by induction we are done.
\end{itemize}
\end{proof}

\begin{prop}\label{prop:1first}
Let $F \in \mathcal{F}$ be a Generalized Flowerbed of size $n$ such that no pair of Flowers cancels and there is a Generalized Flower of height at least $2$. Suppose that the first player has the weakest Generalized Flower. Then under both play conventions the winning player can win by first cutting down, in any order, all of their opponent's Generalized Flowers with height $1$. 
\end{prop}

\begin{proof}
Assume without loss of generality that Left wins.

If there is a red Flower of height at least $2$, then by Lemma \ref{lem:1ww} Left can guarantee that at each turn, there is either at least one Flower of height at least $2$, or at most one Flower. Thus by Corollary \ref{cor:1w,>2}, the first player plays to maximize the nim-sum of the Stalks, and the second player to minimize the nim-sum. We note in particular that cutting down Generalized Flowers of height $1$ does not change the nim-sum of the Stalks. Hence, by Lemma \ref{lem:mixUD} and the fact that the upper and lower nim-sum are increasing in the first variable, we see that given any order in which the second player cuts down Generalized Flowers, the first player can maximize the nim-sum by first cutting down their opponent's Generalized Flowers of height $1$. Similarly, the second player can minimize the nim-sum by first cutting down their opponent's Generalized Flowers of height $1$. Hence Left can win by first cutting down all the red Generalized Flowers of height $1$.

If there is no red Flower of height at least $2$, the red Flowers are all of height $1$ and 
we want to show that Left can win by continually cutting down Flowers. As there is a Flower of height at least $2$, there is a blue Flower of height at least $2$. Thus Left plays first and Right plays second. Since Right cannot win by leaving a blue Flower of height at least $2$ until last, we must have, for any permutation $\sigma$ with $b_{\sigma(n)} \geq 2$, 
\begin{eqnarray*}
a ~\text{\sout{$\downarrow$}}~ (b_{\sigma(1)}-1) ~\text{\sout{$\downarrow$}}~ \cdots ~\text{\sout{$\downarrow$}}~ (b_{\sigma(n-1)}-1) \geq (b_{\sigma(n)}-1) \geq 2.
\end{eqnarray*}
This tells us that $a \geq 2$, so that there is a Stalk of height at least $2$, and
\begin{eqnarray*}
a ~\text{\sout{$\downarrow$}}~ (b_{\sigma(1)}-1) ~\text{\sout{$\downarrow$}}~ \cdots ~\text{\sout{$\downarrow$}}~ (b_{\sigma(n-1)}-1) ~\text{\sout{$\downarrow$}}~  (b_{\sigma(n)}-1) > 0.
\end{eqnarray*}
Therefore Left can win by continually cutting down red Flowers. 
\end{proof}

We are now ready to prove that Theorem \ref{thm:main} holds for Generalized Flowerbeds without canceling pairs that are not Sprig games.\\

\begin{thm}\label{thm:o+o-}
Let $F \in \mathcal{F}$ be a Generalized Flowerbed of size $n$ such that no pair of Generalized Flowers cancels and there is a Flower of height at least $2$. Then $o^+(F)=o^-(F)$. Moreover, if $n \geq 2$, then the winning player can make their next move in such a way that allows him to win under both normal and mis\`ere play.
\end{thm}

\begin{proof}

We prove the theorem by induction on $n$, noting that the base case $n=1$ is true by Propositions \ref{prop:n1,b<c} and \ref{prop:n1,b=c}.

For $n \geq 2$, assume without loss of generality that Left wins. 

If the second player has the weakest Flower, by Theorem \ref{thm:1win}, the first player has a move that is winning under both normal and mis\`ere play. So suppose that the first player has the weakest Flower. Note that by Corollary \ref{cor:2play}, the first player must cut down one of the opponent's Flowers, and the second player must do likewise, so that after each player has made one move the resulting game is a Generalized Flowerbed of size $n-1$.

Suppose that there is a red Generalized Flower of height $1$. Then by Proposition \ref{prop:1first}, Left can win under both normal and mis\`ere play by first cutting down any given red Generalized Flower of height $1$. So suppose that there are no red Generalized Flowers of height $1$. Then after each player has made one move the resulting game is a Generalized Flowerbed of size $n-1$ with a Flower of height at least $2$.

If Left plays first and wins under normal play, then there exists a left follower $F^L$ such that for any follower $F^{LR}$ of $F^L$, we have $F^{LR} \in \mathcal{N}^+ \cup \mathcal{L}^+$. Since we may assume that $F^{LR}$ is a Generalized Flowerbed of size $n-1$ with a Flower of height at least $2$, by induction for any follower $F^{LR}$ of $F^L$ we also have that $F^{LR} \in \mathcal{N}^- \cup \mathcal{L}^-$, so that Left wins under mis\`ere play by moving to the same follower $F^L$. 

If Left plays second and wins under normal play, then for any right follower $F^R$ of $F$, there exists a left follower $F^{RL} \in \mathcal{P}^+ \cup \mathcal{L}^+$ of $F^R$. Since we may assume that $F^{RL}$ is a Generalized Flowerbed of size $n-1$ with a Flower of height at least $2$, by induction $F^{RL} \in \mathcal{N}^- \cup \mathcal{L}^-$. Hence if Right moves to $F^R$, Left can win under both play conventions by moving to $F^{RL}$.

\end{proof}

\subsection{Canceling Flowers}\label{cancel}
Given a Flowerbed $F$, let the \textit{trimmed form of $F$} be the minimal Flowerbed $F_t$ that can obtained by removing pairs of Flowers that cancel. 

We first classify the outcome classes for Generalized Flowerbeds where the trimmed form is a Generalized Sprigs game, and thus prove that Theorem \ref{thm:main} holds for these games. We then prove that Theorem \ref{thm:main} holds for all Generalized Flowerbeds where the trimmed form is not a Generalized Sprigs game. 

%THEOREM ABOUT o+=o- all unequals are height 1
%Needs polishing
%CHANGED SLIGHTLY TO INCLUDE F*
\begin{thm}\label{thm:special}
\vspace{5pt}
Let $F \in \mathcal{F}$ be a Generalized Flowerbed with trimmed form $F_t$, where $F_t$ is composed of $n_t$ Generalized Sprigs of each color and Stalks all of height $1$, and has no pairs of canceling Flowers. Let $\{x_i\}$ and $\{y_j\}$ be the blossom values of the blue and red Generalized Flowers in $F_t$ respectively, and let $\epsilon(F)= \min\{x_i\}-\min\{y_j\}$. If $F_t$ is empty, let $\epsilon(F)=0$. Let $a$ be the number of Stalks in $F$. Then 
\begin{eqnarray*}
o^+(F)=o^-(F)=o^+(F_t) = o^-(F_t+*)= \left\{ 
\begin{array}{cc}
\mathcal{N} & \text{ if } a \neq 0  \\
\mathcal{P} & \text{ if } \epsilon(F)=0 \text{ and } a=0\\
\mathcal{L} & \text{ if } \epsilon(F)>0 \text{ and } a=0\\
\mathcal{R} & \text{ if } \epsilon(F)<0 \text{ and } a=0.
\end{array}
\right.
\end{eqnarray*}
\end{thm}

\begin{proof}
We note first that by choice of $F$, $ F+* = F^*$ and $F_t+* = F_t^*$. We note also that the equality $o^+(F) = o^+(F_t)$ is trivial, and the equality of $o^+(F_t)$, $o^-(F_t^*)$ and the expression on the right is true by Theorem \ref{thm:Generalized Sprigs}. Hence it suffices to prove that $o^-(F)$ is equal to the expression on the right. 

Suppose $a = 0$. Then since all Stalks have height $1$, and we may cancel $*+*$, this means that we may assume that $F$ has no Stalks. As Right plays optimally, we may assume that if there are more blue flowers than red flowers and it is Right's turn, he will cut in the stem of a blue flower.

If $\epsilon (F) = 0$, then $F_t$ is empty, and as we may assume that there are no pairs of canceling Sprigs, the game consists only of pairs of canceling Generalized Flowers of height at least $2$. Suppose that Left plays second. Then she wins using the following strategy. If Right plays in a Blossom, Left responds in one of the Blossoms in the canceling pair so as to keep the sum of their blossom values non-negative. If Right plays in a stem, Left responds in the stem of the other flower in the canceling pair so as to maintain a favorable nim game in the Stalks and stems. Eventually the game reduces to a game of nim that Left wins. 

If $\epsilon(F) > 0$, Left wins by using the following strategy. If Right plays in a stem or Blossom of a canceling flower, Left responds in the canceling pair. Otherwise, if there are at least $2$ red flowers remaining in $F_t$, she cuts in the stem of the one with the largest blossom value, and if there is only $1$ red flower remaining in $F_t$, she plays in the blossom of a blue flower remaining in $F_t$ so that sum of the blossom values of the two flowers is non-negative. She is able to do so, because if there is only $1$ red flower remaining in $F_t$, there is only $1$ blue flower remaining in $F_t$, and they have blossom values $x'$ and $y_1$ respectively, where $x' \geq x_1 >y_1$. Hence eventually the game is of the form $E + *:(x') + *:(-y') + A$, where $E$ consists of pairs of canceling flowers of height at least $1$, $x'+y' \geq 0$, $A \in \mathcal{P}^-$ is a stalks game and Right plays next. In this game, Left is able to win playing second by using the same strategy as when $\epsilon (F) = 0$. Similarly, if $\epsilon(F)<0$, Right wins.

If $a \neq 0$, then we may assume that $F$ has exactly one Stalk. If $\epsilon(F) = 0$, then the next player wins by cutting the Stalk. 

If $\epsilon(F) < 0$, we show that both players have a winning move. If Right plays first, he cuts down the Stalk and wins. If Left plays first, she wins by using the following strategy. She first cuts down one of the red flowers in $F_t$. If Right plays in a stem of a canceling flower, Left responds in the canceling pair so as to maintain a favorable nim game in the Stalks and stems. Otherwise, if there is at least one red flower remaining in $F_t$, she cuts one down, and if not, she cuts the Stalk. She is able to do so since at each turn Right must play in a stem of a blue flower. Hence eventually Left cuts the Stalk, and the game is a Stalks game either with zero nim-sum  and some stalk of height at least $2$, or non-zero nim-sum and all stalks of height $1$. In either case, the previous player wins, so Left wins. Similarly, if $\epsilon(F)>0$, both players have a winning move.
\end{proof}

In order to prove Theorem \ref{thm:main} for Generalized Flowerbeds where the trimmed form is not a Generalized Sprigs game, we first prove an intermediary result.

%lem:n1,b=c,x=y
\begin{lem}\label{lem:n1,b=c,x=y} Let $F \in \mathcal{F}$ be a Generalized Flowerbed where the Generalized Flowers form canceling pairs. Let $F_L$ and $F_R$ be the sum of all the blue Generalized Flowers and all the red Generalized Flowers respectively. Let $F_t$ be the Stalks game that is the trimmed form of $F$, and let it have nim-sum $a$. Then 
\begin{eqnarray*}
o^+(F) = o^-(F^*) = o^+(F_t)= o^-((F_t)^*) = \left\{
\begin{array}{cc}
\mathcal{N} & a \neq 0 \\
\mathcal{P} & a = 0 .
\end{array}
\right.
\end{eqnarray*}
\end{lem}
\begin{proof}
We note that the final two equalities hold by our previous results for Generalized Flowerbeds without canceling pairs (Proposition \ref{prop:n1,b<c} and Proposition \ref{prop:n1,b=c}), and the equality of $o^+(F)$ and $o^+(F_t)$ holds as canceling flowers are equivalent under normal play. Thus it remains to prove that $o^-(F^*)$ is equal to the expression on the right.
 
We note first that if there are no Generalized Flowers in $F_L$ or $F_R$, the result is trivially true, so we may assume that $F_L+F_R$ is non-empty. In addition, if there is a pair of Generalized Sprigs that cancel, then by Theorem \ref{thm:*equal}, removing the pair does not affect the outcome of the game. Thus we can assume that all Flowers in $F_L+F_R$ have height at least $2$. This means, in particular, that $F^*=F$.

Consider the game $F$ and assume without loss of generality that Right plays first. Suppose that $a = 0$. If Right plays in a Blossom, Left responds in the Blossom of the canceling Flower to keep the sum of their normal play values non-negative. If Right plays in a Stalk or stem, then Left plays in a Stalk or the stem of the canceling Flower respectively in order to maintain a favorable nim game in the Stalks and stems. Eventually the game reduces to a game of nim that Left wins. Hence if Right plays first, Left wins, and similarly if Left plays first, Right wins.

If $a \neq 0$, the first player wins in $F$ by making the Stalk nim-sum $0$.
\end{proof}

We call a game $F_L+F_R$ \textit{left-canceling} if $F_L$ consists of blue Flowers, $F_R$ consists of the same number of red Flowers, all the Flowers are of height at least $2$, and for each $i$, the $i$th weakest  blue Flower is both the same height as and at least as strong as the $i$th weakest red Flower. 

%THEOREM ABOUT o+=o- ii getting rid of equally weak
\begin{thm}
Suppose $F \in \mathcal{F}$ is a Generalized Flowerbed of size $n$, and its trimmed form $F_t$ is a Generalized Flowerbed of size $n_t$. Then 
\vspace{-5pt}
\begin{eqnarray*}
o^+(F)=o^-(F^*)=o^+(F_t)=o^-((F_t)^*).
\end{eqnarray*}
\end{thm}
%pf of ii: getting rid of equally weak
\begin{proof}
Let the pairs of Flowers that cancel be $F_L+F_R$, where $F_L$ contains the blue ones and $F_R$ the red ones. Then $F=F_L+F_R+F_t$.

As in Lemma \ref{lem:n1,b=c,x=y}, we can assume that all Flowers in $F_L+F_R$ have height at least $2$. We also note that equalities $o^+(F)=o^+(F_t)=o^-((F_t)^*)$ are trivial by the fact that $F_L+F_R =^+0$ and by Theorem \ref{thm:o+o-}. Moreover, by Theorem \ref{thm:special}, the equalities all hold for Generalized Flowerbeds where the trimmed form is a Generalized Sprigs game, that is, when all the Generalized Flowers in $F_t$ have height $1$. Thus, it suffices to prove that $o^-(F^*)$ is equal to any of the other three.

We assume mis\`ere play and prove a slightly stronger statement by induction. \vspace{5pt}\\
\textit{Inductive Hypothesis: Let $F=F_L+F_R+F'$, where $F_L+F_R$ is Left-canceling and $F'$ is a Generalized Flowerbed of size $n'$ with a Generalized Flower or Stalk of height at least $2$. Suppose that Left has a winning strategy under mis\`ere play for $F'$. Then Left has a winning strategy under mis\`ere play for $F$.}\vspace{5pt}\\
In the following we let the \textit{paired Flower} of the $i$th weakest Flower in $F_L$ be the $i$th weakest Flower in $F_R$, and vice versa, where pairings do not change as the game progresses. We induct on $n'$. \vspace{5pt}\\
\textit{Base Cases: $n'=0$ and $n'=1$}

If $n'=0$, suppose $F_L+F_R$ is comprised of canceling pairs. Then the hypothesis is true by Lemma \ref{lem:n1,b=c,x=y}. Otherwise, for each $i$, the $i$th weakest blue flower is $*_{h_i}:(x_i)$ and the $i$th weakest red flower if $*_{h_i}:(-y_i)$ for some $h_i$ and $x_i \geq y_i$. Again by Lemma \ref{lem:n1,b=c,x=y}, Left wins $H=\sum_i *_{h_i}:(y_i)+ \sum_i *_{h_i}:(-y_i) + F'$. Hence Left can win in $F$ by following the same strategy as in $H$, at each turn keeping the sum of blossom values $x_i-y_i$ at least as large as the sum of the corresponding blossom values in $H$. Hence if Left has a move in $H$, Left can perform an equivalent move in $F$, and so Right must make the last move and Left wins.

If $n'=1$, we let $F' = *_{b_1}:(x_1) + *_{c_1}:(-y_1)+A$, where $A$ is a Stalks game with nim-sum $a$. We split this into a number of cases.

Suppose that all Stalks in $F'$ have height $1$. In particular, this means that $F'$ has a Generalized Flower of height at least $2$. Then we have the following cases:
\begin{itemize}
%b1<c1 Left first
\item\textit{Case 1: $b_1<c_1$, Left plays first}
\begin{itemize}
\item \textit{Sub-case a: $b_1=1$}\\
Left cuts the red Generalized Flower in $F'$ to a Stalk of height $0$ or $1$ to leave the game $F_L+F_R+*:(x_1) + *$ with Right playing next. Since by Lemma \ref{lem:n1,b=c,x=y} $F_L+F_R+*+* \in \mathcal{P} \cup \mathcal{L}$, Left wins playing second in $F_L+F_R+*:(x_1) + *$ by Lemma \ref{lem:ordgreater}, with $H=(x_1)$, $G=*$ and $Y=F_L+F_R+*$. 
\item \textit{Sub-case b(i): $b_1 >1$, no Stalks}\\
Left cuts the red Generalized Flower to a Stalk of height $b_1$ to leave a game $F_L+F_R+*_{b_1}:(x_1) + *_{b_1}$ with Right playing next. Since by Lemma \ref{lem:n1,b=c,x=y} $F_L+F_R+*_{b_1}+*_{b_1} \in \mathcal{P} \cup \mathcal{L}$ , Left wins as in the first subcase by Lemma \ref{lem:ordgreater}.
\item \textit{Sub-case b(ii): $b_1>1$, one Stalk}\\
Left's winning strategy in $F$ is as follows. First, she cuts the Stalk. Then Right must cut down a blue Generalized Flower, or else they move to Sub-case b(i) and Left wins.

If Right plays in a Generalized Flower in $F_L+F_R$, Left mimics in the paired Generalized Flower in $F_L+F_R$, otherwise she follows her winning strategy in $F'$. Thus at each turn Right must cut down a Generalized Flower. 

If all of $F_L+F_R$ is cut down before Right plays in $F'$, then it is Right's turn and the remaining game is of the form $*_{b_1}:(x_1) + *_{c_1}:(-y_1) + A$, where $b_1<c_1$ and $A$ is the Stalks game with nim-sum $0$ comprised of the uncut portions of the stems in $F_L+F_R$. Hence, by Proposition \ref{prop:n1,b<c} Left wins. 

If Right plays in $F'$ before all of $F_L+F_R$ is cut down, this must be to cut the blue Generalized Flower in $F'$ to a Stalk of height $b_1'<b_1$. Left then cuts the red Generalized Flower to a Stalk of height $b_1'$ to leave a game $F_L'+F_R'+*_{b_1'}+*_{b_1'}$ where $F_L'+F_R'$ is left-canceling. Thus Left wins playing second by the case $n'=0$.
\end{itemize}
%b1<c1 Right first
\item\textit{Case 2: $b_1<c_1$, Right plays first}\\
Note first that since all Stalks in $F'$ have height $1$ and $a \leq 1$, Right must first cut down a blue Generalized Flower, or else we move to the same case with Left playing first and Left wins. We give a strategy for Left to win $F$.

If Right cuts down a Generalized Flower in $F_L$ to a Stalk of height $h \geq 2$, Left cuts down the paired Flower in $F_R$ to a Stalk of height $h'$ for which she has a winning strategy playing second in $F' + *_h + *_{h'}$. By Proposition \ref{prop:n1,b<c}, at least one of $0$, $1$ or $h$ will work as a value for $h'$. Then since there is now a Stalk of height at least two, Left wins.

If Right cuts down a Generalized Flower in $F_L+F_R$ to a Stalk of height $0$ or $1$, Left mimics in the paired Generalized Flower in $F_L+F_R$, otherwise she follows her winning strategy in $F'$. Thus at each turn Right must cut down a Generalized Flower. 

If all of $F_L+F_R$ is cut down before Right plays in $F'$, then it is Right's turn and the remaining game is of the form $F' + A$, where $A$ is the Stalks game with nim-sum $0$ comprised of the uncut portions of the stems in $F_L+F_R$, so by Proposition \ref{prop:n1,b<c} Left wins. 

If Right plays in $F'$ before all of $F_L+F_R$ is cut down, this must be to cut the blue Generalized Flower in $F'$ to a Stalk of height $b_1'<b_1$. Left then cuts the red Generalized Flower to a Stalk of height $b_1'$ to leave a game $F_L'+F_R'+A'$ where $F_L'+F_R'$ is left-canceling and $A'$ has nim-sum $0$. Thus Left wins playing second by Lemma \ref{lem:n1,b=c,x=y}.

%b1=c1 x1>y1 Left first
\item\textit{Case 3: $b_1=c_1\geq 2$, $x_1>y_1$, Left plays first}\\
Let $\alpha$ be the largest integer such that $2^{\alpha}|b_1$. Since all the Stalks have height $1$, they have nim-sum $a \leq 1$ so $a < 2^{\alpha+1}$ and, by Proposition \ref{prop:n1,b=c}, Left wins $F'$. So we give a strategy for Left to win $F$.

Let $G$ be the same game $F$ with the Blossoms $x_1$ and $y_1$ replaced by empty Blossoms. Since $G$ is symmetric, $G$ is either next player win or previous player win.

Suppose $G \in \mathcal{N^-}$. Then Left simply follows her winning strategy in the modified game, with the one and only difference that if Right plays in the Blossom $x_1$ or $y_1$, Left responds so as to keep the sum of the Blossom normal-play values non-negative. 

So suppose $G \in \mathcal{P^-}$. Then Left plays in the Blossoms $x_1$ and  $y_1$ so as to keep the sum of their normal play values non-negative and then follows her winning strategy for $G$, with the one and only difference that if Right plays in the Blossom $x_1$ or $y_1$, Left responds so as to keep the Blossom sum non-negative.
%b1=c1 x1>y1 Right first
\item\textit{Case 4: $b_1=c_1\geq 2$, $x_1>y_1$, Right plays first}\\
For the same reasons as in Case 2, Right must always cut a blue Generalized Flower. Again, if he cuts down a Generalized Flower in $F_L$ to a Stalk of height at least $2$, then Left has a winning move, and if he cuts a blue Generalized Flower in $F_L$, he must cut it to a Stalk of height $1$, and Left can mimic this move. So in this case as well, eventually Right cuts the blue Generalized Flower in $F'$ to a Stalk of height $b_1'<b_1$, and Left cuts the red Generalized Flower in $F'$ to a Stalk of height $b_1'$, $0$ or $1$, at least one of which will leave a winning Stalk position for her.
\end{itemize}

Thus the inductive hypothesis holds in the case when all the Stalks in $F'$ have height $1$.

Suppose now that there is a Stalk of height at least $2$ in $F'$. Then Left's winning strategy in $F$ is as follows. If Right has just played in a Blossom in $F_L+F_R$, then Left responds either in the same Blossom or in the Blossom of its paired Flower so as to keep the sum of their normal play values non-negative. If Right has just played in a stem in $F_L+F_R$, Left responds by performing the same move in the paired Flower. Otherwise, Left follows her original strategy in $F'$. 

If in this process the remaining Stalks all have height $1$, then the game is of the form $F'_L+F'_R+F''$, where $F''$ is a Generalized Flowerbed of size $1$ with Stalks all of height $1$, and Left has a winning strategy for $F''$. Thus, by the above analysis, Left wins.

Otherwise, the game reaches a position where both players have cut down exactly one Generalized Flower in $F'$, the game is of the form $F'_L + F'_R + A'$, where $F'_L +F'_R$ is Left-canceling, it is Right's turn, and Left has a winning strategy playing second in $A'$. By the case $n'=0$, Left wins.

\textit{Inductive Step}\\
Suppose that the statement is true for $n'-1$. Then Left's winning strategy in $F$ is as follows. At any move, if Right has just played in $F_L+F_R$, then Left responds in the  paired Flower in $F_L+F_R$. Otherwise, she plays the move in $F'$ that allows her to win under both normal and mis\`ere play. Eventually both players have cut down exactly one Generalized Flower in $F'$, and reach a game of the form 
\vspace{-10pt}
\begin{eqnarray*}
F'_L + F'_R + A' + F'',
\end{eqnarray*}
where $F'_L+F'_R$ is left-canceling, $A'$ is a Stalk game with nim-sum $0$ comprised of the uncut portions of the cut stems in $F_L+F_R$, and $F''$ is the result of cutting down the two Generalized Flowers from $F'$. In particular, $F''$ has $n'-1$ Generalized Flowers of each color, no blue Generalized Flower and red Generalized Flower equally weak, and following the initial order of play, Left wins in $F''$. Thus, since play in a game of Stalks in normal play depends only on the nim-sum of the Stalks, $o^+(F'_L+F'_R+A'+F'' )= o^+(F'_L+F'_R+F'')$.

Moreover, by Lemma \ref{lem:1ww}, as $F'$ has a Generalized Flower of height at least $2$, $F''$ has a Generalized Flower of height at least 2. Hence, we can assume that the inductive hypothesis is true for both $F''$ and $F''+A'$. In particular, by induction, we have $o^+(F'_L+F'_R+(A'+F'')) = o^-(F'_L+F'_R+(A'+F''))$ and $o^+(F'_L+F'_R+F'') = o^+(F'')=o^-(F'')$, so combining these results we obtain that
\begin{eqnarray*}
o^+(F'_L+F'_R+A'+F'') = o^-(F'_L+F'_R+A'+F'')=o^+(F'_L+F'_R+F'') = o^+(F'')=o^-(F'').
\end{eqnarray*}
Thus, following the initial order of play, Left wins from this new position. If Left follows the above strategy, then the order of play when they reach this position is the same as the initial order, and so Left wins in $F$ under mis\`ere play.
\end{proof}

\section{Mis\`ere Star-Based Hackenbush Positions}
Our results show that for Hackenbush positions in which a green Stalk supports either a Green Hackenbush position or a Red-Blue Hackenbush position, there is a simple relationship between the outcome classes under both play conventions. While a similar result for general star-based positions would be difficult to obtain (consider, for example, positions consisting of a green Stalk supporting a disjunctive sum of Flowers!), our results can be extended to any single star-based Hackenbush position $*:G$.

\begin{thm}
Let $G$ be a star-based Hackenbush position. Let $G^* = G+*$ if the graph corresponding to $G$ with shortest height has height $1$, and $G^*=G$ otherwise. Then
\begin{eqnarray*}
o^+(G) = o^-(G^*), \text{ and } o^-(G) = o^+(G^*).
\end{eqnarray*}
\end{thm}
\begin{proof}
We first prove that $o^+(G) = o^-(G^*)$. As the first player can win $G$ under normal play by cutting the grounded green edge, $o^+(G)=\mathcal{N}$.

If $G^*=G+*$, take a graph representing $G$ with height $1$. Then the first player can win $G^*$ under mis\`ere play by cutting the grounded green edge of $G$, leaving $*$, so $o^-(G)=\mathcal{N}$. If $G^*=G$, then $G=*_n:\hat G$ for some $n \geq 2$ and position $\hat G$, so that the first player can win $G^*$ under mis\`ere play by cutting the green edge right above the grounded green edge of $G$, leaving $*$, so $o^-(G)=\mathcal{N}$ again.

We now prove that $o^-(G) = o^+(G^*)$. As $G$ is a star-based Hackenbush position, $G=*:\hat G$ for some Hackenbush position $\hat G$, and so $o^-(G) = o^+(\hat G)$. It suffices to prove that $o^+((*:H)^*)=o^+(H)$ for any Hackenbush position $H$.

We note first that for any Hackenbush position $H$, $o^+(*:H + *) = o^+(H)$. This is because, in the game $*:H+*$, if any player cuts one of the grounded green edges, the other player cuts the other one and wins. Hence the last player able to move in $H$ wins. Thus, if $(*:H)^*=*:H+*$, we are done. Otherwise, $(*:H)^*=*:H$, and so $H$ is also a star-based position. Hence $o^+((*:H)^*)=o^+(H) = \mathcal{N}$. 
\end{proof}

%%%%%%%%%%%%%%%%
%It is well known that the outcome under normal play of a game $G$ is equal to the outcome under mis\`ere play of the game $H$ obtained by attaching all the grounded vertices of $G$ to the top of a single green edge. 

%%%%%%%%%%%%%%%%

We conclude this section with the question: For which disjunctive sums of star-based Hackenbush positions $G$ do there exist evil twins $G^* \in \{G,G+*\}$ such that $o^+(G) = o^-(G^*)$ and $o^+(G^*)=o^-(G)$?

\section{Future Directions}
%Define 'evil twin' here?
%Note that results generalize previous results 
We have related the normal and mis\`ere play outcome classes for some star-based Hackenbush positions and completely classified certain classes of positions. In particular, we have shown that if $G$ is a disjunctive sum of Shrubs and Generalized Flowers, then it has an evil twin $G^*$ in the set $\{G, G+*\}$, such that the outcomes of $G$ and $G^*$ are equal after toggling play conventions. Moreover, if $H$ is the Stalks representation of underlying nim game in the stems and Shrubs, $G^*$ is obtained by adding $*$ to $G$ if and only if $H^*$ is obtained by adding $*$ to $H$. From this it follows that a winning strategy under mis\`ere play is to play as in normal play until the only winning move is to a game $G'$ where the underlying nim game in the stems and Shrubs consists only of piles of size $1$, after which a winning strategy can be easily determined by examining the advantage and edges of followers of $G'$, as in Section \ref{Sprigs}. We conclude by listing some possible directions for future investigation.

\begin{enumerate}
\item We have shown that if $G$ is a Shrub, then it is equivalent to a nim-heap, and if $G=*_n:B$ is a Generalized Flower with blossom value $x$, then it can be treated in the Generalized Flower universe as the game $*_n:x$. Can we say more about equivalences of Generalized Flowers? For an integer $n \geq 1$ and real $x$, which Generalized Flowers are equivalent to $*_n:x$ in the dicot universe? In the universe of star-based Hackenbush positions? In the universe of all mis\`ere games?
\item We have shown that, for all disjunctive sums of star-based Hackenbush positions $G$ where the blossom is either all green or all red and blue, there exists an evil twin $G^* \in \{G,G+*\}$ such that $o^+(G) = o^-(G^*)$ and $o^+(G^*)=o^-(G)$. This partially answers a question of McKey et al., which asks, for which Hackenbush positions do we have $o^+(G)=o^-(G+*)$ and $o^+(G+*)=o^-(G)$? We further ask whether there exists such $G^*$ for all disjunctive sums of star-based Hackenbush positions, and other Hackenbush position in general.
\item We have shown that solving Hackenbush Flowers under one play convention is equivalent to solving it under the other play convention. However, Berlekamp's original question remains open: Who wins a game of Hackenbush Flowers under normal play?
\item We have shown that under mis\`ere play all star-based Green Hackenbush positions, in particular Green Hackenbush positions without cycles, are equivalent to nim heaps. What about general Green Hackenbush positions? How would one deal with cycles? Are there any relationships between the mis\`ere play nim-values and outcome classes?
\end{enumerate}

\section{Acknowledgements}

This research was performed at the University of Minnesota Duluth REU run by Joe Gallian. The REU was supported by the National Science Foundation, the Department of Defense (grant number DMS 1062709), the National Security Agency (grant number H98230-11-1-0224), and Princeton University. 

I would like to thank Joe Gallian for supervising this research. I am also grateful to Eric Riedl, Lynnelle Ye, Alison Miller, Adam Hesterberg and the other participants at the 2012 Duluth REU for their insightful comments and support. 

\bibliographystyle{plain}	% (uses file "plain.bst")
\bibliography{HackenbushRefs}		% expects file "myrefs.bib"

\end{document}